%% file: General_position_in_digraphs_arXiv.tex
\documentclass[a4paper,12pt]{article}
\usepackage[left=2.5cm,right=2.5cm,top=2.5cm,bottom=2.5cm]{geometry}
\usepackage{amssymb}
\usepackage{amsthm}
\usepackage{tikz}
\usepackage{url}
\usepackage{tikz,pgfplots}
\usetikzlibrary{decorations.markings}
\usepackage{amsmath,amssymb}
\usepackage{todonotes}
\usetikzlibrary{decorations.markings}
\usepackage{color}
\usepackage{graphicx}
\usepackage{placeins,caption}
\usepackage{float}
\usepackage{hyperref}
\usepackage{xspace}

\hypersetup{colorlinks=true, linkcolor=blue, citecolor=blue, urlcolor=blue}

\newtheorem{theorem}{Theorem}[section]
\newtheorem{lemma}[theorem]{Lemma}
\newtheorem{corollary}[theorem]{Corollary}
\newtheorem{proposition}[theorem]{Proposition}

\newtheorem{conjecture}[theorem]{Conjecture}
\newtheorem{problem}[theorem]{Problem}

\theoremstyle{definition}
\newtheorem{definition}[theorem]{Definition}
\newtheorem{observation}[theorem]{Observation}

\newcommand{\IS}{{\sc Independent Set}\xspace}
\newcommand{\GP}{{\sc Oriented-GP}\xspace}

\newcommand{\ac}{{\rm a}}
\newcommand{\sink }{{\rm Sink}}
\newcommand{\nsink }{{\rm Sink'}}
\newcommand{\source }{{\rm Source}}
\newcommand{\nsource }{{\rm Source'}}

\newcommand{\gp}{{\rm gp}}
\newcommand{\diam}{{\rm diam}}
\newcommand{\diamm}{\rm diam ^*}
\newcommand{\ext}{{\rm Ext}}
\newcommand{\ugp}{{\rm gp^{\rightarrow }}}
\newcommand{\lgp}{{\rm gp_{\rightarrow }}}
\DeclareMathOperator{\Circ}{Circ}
\DeclareMathOperator{\Ka}{Ka}
\DeclareMathOperator{\Pe}{P}
\newcommand{\Z}{\mathbb{Z}}

\tikzset{middlearrow/.style={
		decoration={markings,
			mark= at position 0.7 with {\arrow[scale=2]{#1}} ,
		},
		postaction={decorate}
	}
}

\tikzset{midarrow/.style={
		decoration={markings,
			mark= at position 0.5 with {\arrow[scale=2]{#1}} ,
		},
		postaction={decorate}
	}
}

\let\deg\relax
\DeclareMathOperator {\deg} {deg}

\newcommand{\address}[1]{#1}

\begin{document}
	
	\title{The general position number of digraphs}	
	\author{ Ullas Chandran S.V. $^{a}$ \\ \texttt{\footnotesize svuc.math@gmail.com}        
		\and Gabriele Di Stefano $^{b}$ \\ \texttt{\footnotesize gabriele.distefano@univaq.it}
		\and Grahame Erskine $^{c}$ \\ \texttt{\footnotesize grahame.erskine@open.ac.uk}
		\and
		Haritha S $^{a}$ \\ \texttt{\footnotesize harithasreelatha1994@gmail.com}
		\and Elias John Thomas $^{d}$ \\ \texttt{\footnotesize eliasjohnkalarickal@gmail.com}
		\and James Tuite $^{c,e}$ \\ \texttt{\footnotesize james.t.tuite@open.ac.uk}
	}
	
	\maketitle	
	
	\address{
		\noindent
		$^a$  Department of Mathematics, Mahatma Gandhi College, University of Kerala,  \\Thiruvananthapuram-695004, Kerala, India \\
		$^b$ Department of Information Engineering, Computer Science and Mathematics, University of L'Aquila,  Italy \\
		$^c$ School of Mathematics, Open University, UK\\
		$^d$ Department of Mathematics, Greenshaw High School, Sutton, UK \\
		$^e$ Department of Informatics and Statistics, Klaip\.{e}da University, Lithuania\\
		
	}

	\begin{abstract}
		The general position number for graphs ask for largest vertex subsets $S$ such that no three vertices are contained on a common shortest path. We examine this problem in the setting of directed graphs. We provide bounds for the general position number of digraphs, show that the problem is NP-complete, investigate the problem for some important families of digraphs such as circulant digraphs, Kautz digraphs and permutation digraphs, and study the general position numbers obtained from all orientations of an undirected graph.
	\end{abstract}
	
	\medskip\noindent
	
	{\bf AMS Subj.\ Class:} 05C12; 05C20; 05C69.

	\section{Introduction}\label{sec:intro}

	The initial inspiration of the general position problem was a chessboard puzzle by the 19th century puzzle-maker Henry Dudeney~\cite{dudeney-1917}, which asked how many pawns can be placed on a chessboard without any three lying on a straight line. The problem was later generalised to the setting of graph theory in~\cite{ref1,ref9} by asking for the largest possible subset $S$ of vertices of a graph $G$ such that no three vertices of $S$ are contained in a common shortest path of $G$. The number of vertices in such a largest subset is called the \emph{general position number} of $G$ and is denoted by $\gp (G)$. This problem has a very large literature; for a survey see~\cite{survey}.
	
	Several variations of the general position problem have been explored, including replacing `shortest path' by `induced path'~\cite{ThoChaTui}, considering only shortest paths of bounded length~\cite{KlaRalYer}, replacing the ordinary graph distance by the Steiner distance~\cite{KlaKuzPetYer}, finding smallest maximal general position sets~\cite{lowergp}, sets of edges in general position~\cite{Manuel}, and the mutual-visibility problem, in which for any $u,v \in S$ we require only the existence of at least one shortest $u,v$-path that does not pass through a third vertex of $S$~\cite{DiStefano}. The contribution of this paper is to extend the general position problem to a different type of network, namely directed graphs, in which each link of the network is assigned a direction. 
	
	A directed graph $D$, or digraph for short, consists of a set $V(D)$ of vertices and a set of ordered pairs $A(D)$ called arcs. The \emph{order} of $D$ is $n = |V(D)|$ and the \emph{size} is $m = |A(D)|$. If there is an arc $(u,v)$ in $A(D)$ we write $u \rightarrow v$ and say that $v$ is \emph{adjacent from} $u$ and $u$ is \emph{adjacent to} $v$. We exclude loops $u \rightarrow u$ and multiple arcs $u \rightarrow v$ with the same direction, as they add nothing of interest for our problem, i.e. we consider only \emph{simple} digraphs. However we do allow 2-cycles, i.e. we may have $u \rightarrow v$ and $v \rightarrow u$. A digraph without 2-cycles is an \emph{oriented graph}. An undirected graph $G$ is a set of vertices $V(G)$ together with a collection of unordered pairs $E(G)$ called edges, and we write $u \sim v$ if $uv$ is an edge in $E(G)$.
	
	For a subset $S$ of $V(D)$ the \emph{induced subdigraph} $\langle S \rangle $ of $D$ is the digraph with vertex set $S$ with an arc $u \rightarrow v$ for $u,v \in S$ if and only if $u \rightarrow v$ in $D$. Associated with any digraph $D$ is an underlying undirected graph on the same vertex set in which there is an edge $u \sim v$ if and only if $D$ contains at least one of the arcs $u \rightarrow v$ or $v \rightarrow u$. Conversely, assigning directions to each of the edges of an undirected graph $G$ yields an \emph{orientation} $G^{\rightarrow }$ of $G$. Hence a graph is oriented if and only if it is an orientation of an undirected graph. The digraph $\hat {D}$ obtained by reversing the orientation of all arcs of $D$ is the \emph{converse} of $D$.
	
	The \emph{out-neighbourhood} $N^+(u)$ of a vertex $u$ of $D$ is defined to be the set $\{ v \in V(D):u \rightarrow v\} $ and the \emph{in-neighbourhood} $N^-(u)$ of $u$ is $\{ v \in V(D) : v\rightarrow u\} $. For $S \subseteq V(D)$ we set $N^+(S) = \bigcup _{v\in S}N^+(v)$ and $N^-(S) = \bigcup _{v\in S}N^-(v)$. The \emph{out-degree} $d^+(u)$ and \emph{in-degree} $d^-(u)$ are defined by $d^+(u) = |N^+(u)|$ and $d^-(u) = |N^-(u)|$. A vertex $u$ with out-degree $d^+(u) = 0$ is a \emph{sink}, whereas if $d^-(u) = 0$, then $u$ is a \emph{source}. A vertex $u$ is a \emph{transmitter} if $d^-(u),d^+(u) > 0$ and for any $x \in N^-(u)$ and $y \in N^+(u)$ we have $x \rightarrow y$. A vertex is \emph{extreme} if it is either a source, sink or transmitter, and we denote the set of all extreme vertices of $D$ by $\ext (D)$. A vertex that is incident with just one arc is a \emph{leaf} and the number of leaves in $D$ is $\ell (D)$. In an undirected graph the neighbourhood $N(u)$ of a vertex $u$ is $\{ v \in V(G):v \sim u\} $ and the degree $d(u)$ is $|N(u)|$. A vertex of degree 1 is a leaf and a vertex adjacent to all other vertices of $G$ is a \emph{universal vertex}. 
	
	A directed $u,v$-path in $D$ is a sequence of distinct vertices $u_0 = u,u_1,\ldots ,u_{\ell }=v$, where $u_i \rightarrow u_{i+1}$ for $0 \leq i < \ell $, and $\ell $ is the length of the path. The \emph{complete digraph} $K_n$ is the simple digraph with $n$ vertices with all $2\binom{n}{2}$ possible arcs present. An orientation of a complete graph is a \emph{tournament}. A cycle of length $\ell $ in a digraph $D$ is a sequence $u_0,u_1,\ldots ,u_{\ell -1}$ of distinct vertices such that $u_i \rightarrow u_{i+1}$ for $0 \leq i \leq \ell -2$ and $u_{\ell -1} \rightarrow u_0$. If $D$ has no cycles then it is \emph{acyclic}. The directed cycle $Z_n$ is the digraph with vertex set $\mathbb{Z}_n$ with an arc $i \rightarrow i+1$ for each $i \in \mathbb{Z}_n$. The \emph{join} $D_1 \vee D_2$ of two digraphs $D_1$ and $D_2$ is the digraph formed from the disjoint union of $D_1$ and $D_2$ by adding all possible arcs from $D_1$ to $D_2$ (note that, unlike the undirected case, the order of the join operation for digraphs is important).
	
	The distance $d(u,v)$ between two vertices is the length of a shortest directed $u,v$-path in $D$. We also call a shortest path a \emph{geodesic}. Note that we may have $d(u,v) \neq d(v,u)$, so that the directed distance is not in general a metric. If $W,W'$ are subsets of $V(D)$, then we set $d(W,W')$ to be the length of a shortest path that begins in $W$ and terminates in $W'$. If $D$ contains a directed path from $u$ to $v$, we say that $u$ \emph{reaches} $v$. If there is a pair $u,v$ for which $u$ cannot reach $v$, then we set $d(u,v) = \infty $ (and similarly for subsets $W,W'$ we set $d(W,W') = \infty $ if no $w' \in W'$ is reachable from any $w \in W$). If there is a directed walk from $u$ to $v$ for any pair of distinct vertices of $D$, then $D$ is \emph{strongly connected} or \emph{strong}; then the value of $\max \{ d(u,v):u,v \in V(D)\} $ is the \emph{diameter} $\diam (D)$ of $D$. If $D$ is not strongly connected, then the diameter of $D$ will be infinite. We will denote by $\diam ^*(D)$ the length of a longest geodesic in $D$; observe that $\diam ^*(D)$ is always finite.  
	
	A maximal strong subdigraph is a \emph{strong component} of $D$, and the strong components partition the vertex set $V(D)$. Equivalently, one can introduce an equivalence relation $R$ on $V(D)$ by setting $uRv$ if $u$ reaches $v$ and $v$ reaches $u$, and the strong components are the equivalence classes for $R$. The connected components of the underlying undirected graph of $D$ are called the \emph{weak components} of $D$. The \emph{condensation} of a digraph $D$ is the digraph with the strong components of $D$ as vertices, with an arc between strong components $S_1$ and $S_2$ if and only if there is an arc from $S_1$ to $S_2$ in $D$. Any condensation is acyclic.
	
	A subdigraph $D'$ of $D$ is a digraph such that $V(D') \subseteq V(D)$ and $A(D') \subseteq A(D)$. A subdigraph $D'$ is \emph{isometric} if for any any pair $u,v \in V(D')$ the distance from $u$ to $v$ is the same in $D'$ as it is in $D$. A digraph is \emph{transitive} if $x \rightarrow y$ and $y \rightarrow z$ implies the existence of an arc $x \rightarrow z$. The strong components of a transitive digraph are all complete digraphs. An undirected graph is a \emph{comparability graph} if it admits a transitive orientation. A subdigraph $D'$ of $D$ is \emph{convex} if for any $u,v \in V(D')$ and any shortest $u,v$-path in $D$, all vertices of $P$ lie in $D'$. A subset $S$ of the vertex set of a digraph $D$ is \emph{independent} if there are no arcs between vertices of $S$, and the \emph{independence number} $\alpha (D)$ of $D$ is the number of vertices in a largest independent set of $D$ (with a similar definition for graphs).
	
	If $a,b \in \mathbb{N}$, then we set $[a] = \{ 1, \ldots ,a\} $ and denote by $[a,b]$ the set $\{ n \in \mathbb{N}: a \leq n \leq b\} $. Any such set $[a,b]$ is an \emph{interval}. The falling factorial is $(x)_r = x(x-1)\ldots (x-(r-1))$. We will write an undirected path of order $n$ as $P_n$. If a graph $G$ has a path $P$ that passes through all vertices of $G$, then $P$ is a \emph{Hamilton path} and $G$ is \emph{traceable}. If $X, Y \subseteq V(G)$, $(X,Y)$ represents the edges of $G$ with an endpoint in $X$ and an endpoint in $Y$. A \emph{split} in a graph $G$ is a partition $V(G) = X \cup Y$ such that $G$ contains all possible edges between $X$ and $Y$.
	
	The plan of this paper is as follows. In Section~\ref{sec:preliminary} we give some simple bounds for the general position number of a digraph and characterise the structure of general position sets. Section~\ref{sec:families} explores general position sets in some well known families of digraphs. In Section~\ref{sec:orientation} we consider the general position numbers of all orientations associated with an undirected graph and give some conjectures. In Section~\ref{sec:conclusion} we conclude with some open problems.
	
	\section{Preliminary results}\label{sec:preliminary}
	
	We begin by defining the general position number for directed graphs by analogy with the undirected case. An example of a general position set is shown in Figure~\ref{fig:d2k3n20}.
	
	\begin{definition}
		For a digraph $D$, a subset $S \subseteq V(D)$ is in general position if for any $u,v \in S$, no shortest $u,v$-path in $D$ passes through a vertex in $S - \{ u,v\} $. The number of vertices in a largest general position set of $D$ is the \emph{general position number} $\gp (D)$ of $D$, and such a largest set is a \emph{$\gp $-set} of $D$.
	\end{definition}

	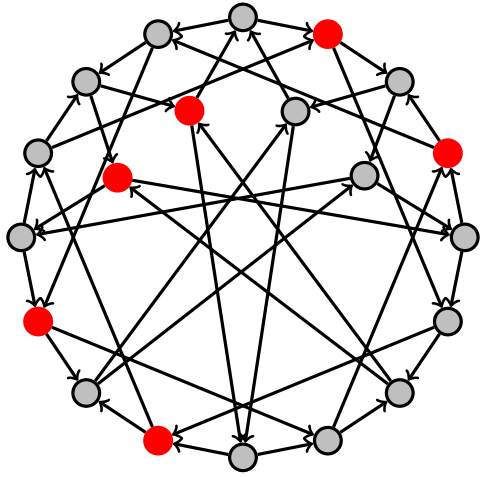
\begin{figure}
		\centering
		\begin{tikzpicture}[x=0.4mm,y=-0.4mm,inner sep=0.2mm,scale=0.25,very thick,vertex/.style={circle,draw,minimum size=10,fill=lightgray}]
			\node at (181.4,173.2) [vertex] (v1) {};
			\node at (222.6,300.4) [vertex,color=red] (v2) {};
			\node at (317.2,211.5) [vertex,color=red] (v3) {};
			\node at (95.9,379.5) [vertex] (v4) {};
			\node at (679.4,379.5) [vertex] (v5) {};
			\node at (387.6,87.8) [vertex] (v6) {};
			\node at (387.6,671.2) [vertex] (v7) {};
			\node at (118.1,267.9) [vertex] (v8) {};
			\node at (118.1,491.1) [vertex,color=red] (v9) {};
			\node at (657.2,267.9) [vertex,color=red] (v10) {};
			\node at (657.2,491.1) [vertex] (v11) {};
			\node at (499.3,110) [vertex,color=red] (v12) {};
			\node at (276,110) [vertex] (v13) {};
			\node at (499.3,649) [vertex] (v14) {};
			\node at (276,649) [vertex,color=red] (v15) {};
			\node at (181.4,585.8) [vertex] (v16) {};
			\node at (593.9,173.2) [vertex] (v17) {};
			\node at (593.9,585.8) [vertex] (v18) {};
			\node at (547.6,297.7) [vertex] (v19) {};
			\node at (456.8,212.6) [vertex] (v20) {};
			\path
			(v1) edge[->] (v2)
			(v1) edge[->] (v3)
			(v8) edge[->] (v1)
			(v13) edge[->] (v1)
			(v2) edge[->] (v4)
			(v2) edge[->] (v5)
			(v18) edge[->] (v2)
			(v3) edge[->] (v6)
			(v3) edge[->] (v7)
			(v18) edge[->] (v3)
			(v4) edge[->] (v8)
			(v4) edge[->] (v9)
			(v19) edge[->] (v4)
			(v5) edge[->] (v10)
			(v5) edge[->] (v11)
			(v19) edge[->] (v5)
			(v6) edge[->] (v12)
			(v6) edge[->] (v13)
			(v20) edge[->] (v6)
			(v7) edge[->] (v14)
			(v7) edge[->] (v15)
			(v20) edge[->] (v7)
			(v8) edge[->] (v12)
			(v15) edge[->] (v8)
			(v13) edge[->] (v9)
			(v9) edge[->] (v14)
			(v9) edge[->] (v16)
			(v10) edge[->] (v13)
			(v14) edge[->] (v10)
			(v10) edge[->] (v17)
			(v12) edge[->] (v11)
			(v11) edge[->] (v15)
			(v11) edge[->] (v18)
			(v12) edge[->] (v17)
			(v14) edge[->] (v18)
			(v15) edge[->] (v16)
			(v16) edge[->] (v19)
			(v16) edge[->] (v20)
			(v17) edge[->] (v19)
			(v17) edge[->] (v20)
			;
		\end{tikzpicture}
		\caption{A $\gp $-set in a digraph (vertices in red)}
		\label{fig:d2k3n20}
	\end{figure}
	
	The recent paper~\cite{digraphs} of Stojanovi\'{c} investigates the related mutual visibility problem in directed graphs. A subset $S \subseteq V(G)$ of a connected undirected graph $G$ is a \emph{mutual visibility set} if for any $u,v \in S$ there exists a shortest $u,v$-path that does not contain any vertex of $S-\{ u,v\} $; and the \emph{mutual visibility number} $\mu (G)$ of $G$ is the number of vertices in a largest mutual visibility set of $G$, see~\cite{DiStefano}. Any general position set of $G$ is also a mutual visibility set, and so $\gp (G) \leq \mu (G)$ for connected graphs. There are two possible definitions of a mutual visibility set for a disconnected graph: either one could define it to be the maximum mutual visibility number over the components, or else the sum of the mutual visibility number of the components. The latter definition ensures that the inequality $\gp (G) \leq \mu (G)$ continues to hold for disconnected graphs. A similar issue arises in directed graphs that are not strongly connected, and we now distinguish the two possible definitions.
	
	\begin{definition}
		\text{ }
		\begin{itemize}
			\item A subset $S \subseteq V(D)$ of a digraph is a \emph{strong mutual visibility set} if for each ordered pair $u,v \in S$ there exists a shortest $u,v$-path in $D$ that does not contain any vertices of $S-\{ u,v\} $.
			\item A subset $S \subseteq V(D)$ is a \emph{weak mutual visibility set} if for each ordered pair $u,v \in S$ either $d(u,v) = \infty $ or there exists a shortest $u,v$-path that does not contain any vertices of $S- \{ u,v\} $.
		\end{itemize}
	\end{definition}
	The `strong' definition is that of Stojanovi\'{c}. For our purposes we prefer the `weak' definition, since it ensures that the inequality $\gp (D) \leq \mu (D)$ holds for all digraphs. Also the `strong' mutual visibility number can be found by taking the maximum of the `weak' mutual visibility number over the strong components of $D$. Observe that the two definitions coincide for strong digraphs. We will not treat mutual visibility further here, except for the following realisation result to compare the two parameters. Note that as the constructions are strong we do not distinguish between strong and weak.
	
	\begin{theorem}
		For any $b \geq a \geq 2$ there exists a strong digraph $D$ with $\gp (D) = a$ and $\mu (D) = b$. 
	\end{theorem}
	\begin{proof}
		As noted above, we have $\gp (D) \leq \mu (D)$ for any digraph. If $a = b$, then the complete digraph $K_a$ suffices. When $b > a$ we use the following construction. Let $K(r)$ be the digraph with vertex set $\{ u_i,v_i : i \in \mathbb{Z}_{r+1}\}$ with all four arcs from $\{ u_i,v_i\} $ to $\{ u_{i+1},v_{i+1}\} $ for $i \in \mathbb{Z}_{r+1}$. For $a \geq 2$ we have $\gp (K(a)) = 2$ and $\mu (K(a)) = a$, so we can assume that $a \geq 3$. Now consider the complete tripartite graph $K_{s,1,1}$ with partite set $\{ x\} $, $\{ y\} $ and $Z = \{ z_1,\ldots ,z_s\} $, where $s \geq 1$.  Orient the edges in $(x,Z)$ towards $Z$, the edges in $(y,Z)$ towards $y$ and the edge $xy$ in the direction $y \rightarrow x$. Finally, for $r \geq 2, s \geq 1$ the graph $K(r,s)$ is formed by identifying the vertex $x$ and the two vertices $u_0,v_0$ from $K(r)$ into a single vertex. An example can be seen in Figure~\ref{fig:K(r,s)}. It is easily verified that $\gp (K(r,s)) = 2+s$ and $\mu (K(r,s)) = r+s$. Hence for $3 \leq a < b$, the digraph $K(b-a+2,a-2)$ has the required parameters. 
	\end{proof}
	
	We note that, whilst showing the existence of an undirected graph with $\gp (G) = a$ and $\mu (G) = b$ for $4 \leq a \leq b$ is not difficult, it is an open problem whether there exist undirected graphs with general position number 3 and arbitrarily large mutual visibility number, so it is interesting that the proof is simple in the directed setting.
	
	\begin{figure}
		\centering
		\begin{tikzpicture}[middlearrow=stealth,x=0.2mm,y=-0.2mm,inner sep=0.2mm,scale=1.4,very thick,vertex/.style={circle,draw,minimum size=10,fill=lightgray}]
			\node at (86.6,50) [vertex,color=blue] (v1) {};
			\node at (0,100) [vertex] (v2) {};
			\node at (-86.6,50) [vertex,color=blue] (v3) {};
			\node at (-86.6,-50) [vertex,color=blue] (v4) {};
			\node at (0,-100) [vertex,color=blue] (v5) {};
			\node at (86.6,-50) [vertex,color=blue] (v6) {};
			
			\node at (51.96,30) [vertex] (u1) {};
			
			\node at (-51.96,30) [vertex] (u3) {};
			\node at (-51.96,-30) [vertex] (u4) {};
			\node at (0,-60) [vertex] (u5) {};
			\node at (51.96,-30) [vertex] (u6) {};
			
			\node at (-100,140) [vertex,color=blue] (z1) {};
			\node at (-50,140) [vertex,color=blue] (z2) {};
			\node at (0,140) [vertex,color=blue] (z3) {};
			\node at (50,140) [vertex,color=blue] (z4) {};
			\node at (100,140) [vertex,color=blue] (z5) {};
			
			\node at (0,180) [vertex] (x) {};
			
			\path
			
			(z1) edge[middlearrow] (x)
			(z2) edge[middlearrow] (x)
			(z3) edge[middlearrow] (x)
			(z4) edge[middlearrow] (x)
			(z5) edge[middlearrow] (x)
			
			(x) edge[middlearrow,bend left] (v2)
			
			(v2) edge[middlearrow] (z1)
			(v2) edge[middlearrow] (z2)
			(v2) edge[middlearrow] (z3)
			(v2) edge[middlearrow] (z4)
			(v2) edge[middlearrow] (z5)
			
			(u1) edge[middlearrow] (v2)
			
			(v1) edge[middlearrow] (v2)

			(v2) edge[middlearrow] (u3)
			(v2) edge[middlearrow] (v3)
			
			(u3) edge[middlearrow] (u4)
			(u3) edge[middlearrow] (v4)
			(v3) edge[middlearrow] (u4)
			(v3) edge[middlearrow] (v4)
			
			(u4) edge[middlearrow] (u5)
			(u4) edge[middlearrow] (v5)
			(v4) edge[middlearrow] (u5)
			(v4) edge[middlearrow] (v5)
			
			(u5) edge[middlearrow] (u6)
			(u5) edge[middlearrow] (v6)
			(v5) edge[middlearrow] (u6)
			(v5) edge[middlearrow] (v6)
			
			(u6) edge[middlearrow] (u1)
			(u6) edge[middlearrow] (v1)
			(v6) edge[middlearrow] (u1)
			(v6) edge[middlearrow] (v1)
			
			;
		\end{tikzpicture}
		\caption{The digraph $K(5,5)$ (with a largest mutual visibility set in blue)}
		\label{fig:K(r,s)}
	\end{figure}
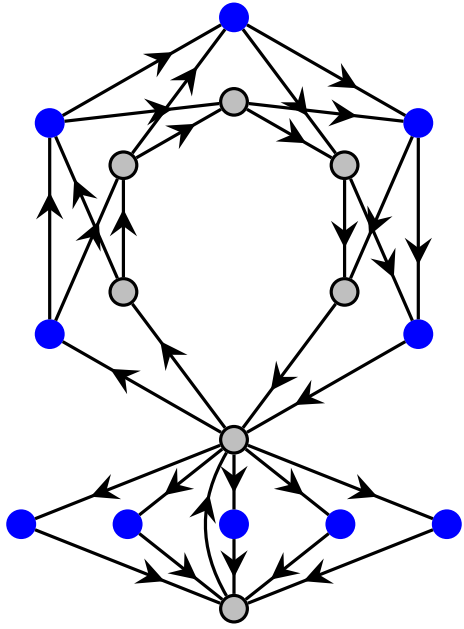
	
	We now give a trivial upper and lower bound on the general position number of digraphs. 
	
	\begin{proposition}\label{prop1}
		If $D$ is a digraph with order $n \geq 2$, then $2 \leq \gp (D) \leq n$. 
	\end{proposition}
	
	There are very few undirected graphs with general position number 2 ($C_4$, $2K_1$ and the path graphs $P_n$ for $n \geq 2$), but, as we will see in Section~\ref{sec:orientation}, any graph with a Hamilton path has an orientation with general position number 2. Similarly, an undirected graph $G$ satisfies $\gp (G) = n$ if and only if $G$ is a complete graph or disjoint union of complete graphs, but whilst undirected graphs meeting the upper bound are scarce, there are many digraphs in which all vertices are contained in a $\gp $-set.
	
	\begin{proposition}\label{prop:gpsetsaretransitive}
		If $S$ is a general position set of a digraph $D$, then the induced subdigraph $\langle S\rangle$ is transitive.
	\end{proposition}
	\begin{proof}
		If $\langle S\rangle$ is not transitive, then $S$ contains vertices $x,y,z$ such that $D$ has arcs $x \rightarrow y$ and $y \rightarrow z$, but does not contain the arc $x \rightarrow z$. This would imply that the directed path $x,y,z$ is a shortest path in $D$ containing three vertices of $S$, a contradiction.
	\end{proof}
	
	\begin{corollary}\label{cor2.5}
		For any digraph $D$ with order $n \geq 2$, $\gp (D) = n$ if and only if $D$ is transitive.
	\end{corollary}
	\begin{proof}
		That $\gp (D) = n$ implies that $D$ is transitive follows from Proposition~\ref{prop:gpsetsaretransitive}. Conversely, if $D$ is transitive, then $D$ has no geodesic with length greater than 1, and so $V(D)$ is in general position. 
	\end{proof}
	
	We now provide a couple of elementary observations on general position sets in a digraph $D$, before building up to a characterisation of the structure of general position sets in digraphs. Firstly, no extreme vertex of $D$ can be an internal vertex of a geodesic of $D$, so that the set of all extreme vertices of $D$ is in general position.
	
	\begin{lemma}\label{lem:ext}
		For any digraph $D$, $\gp (D) \geq |\ext (D)|$.
	\end{lemma}
	
	We can sometimes improve on the bound of Lemma~\ref{lem:ext} using the neighbours of extreme vertices. We define a \emph{near-sink} of $D$ to be a vertex that is not a sink, but all of its out-neighbours are sinks, and similarly a vertex is a \emph{near-source} if it is not a source, but all of its in-neighbours are sources. By $\sink (D)$, $\nsink (D)$, $\source (D)$ and $\nsource (D)$ we denote respectively the sets of sinks, near-sinks, sources, and near-sources of $D$. It is easily seen that the following unions are in general position.
	
	\begin{lemma}\label{lem:sourcesandsinks}
		For any digraph $D$, the sets $\nsink (D) \cup \sink (D)$ and $\nsource (D) \cup \source (D)$ are in general position.
	\end{lemma}
	
	If $S$ is in general position in $D$, then it remains in general position if we reverse all of the arcs.
	
	\begin{observation}\label{obs 1}
		Let $D$ be a digraph and let $\hat{D}$ be its converse. Then $\gp(D) = \gp(\hat{D})$.
	\end{observation}
	\begin{proof}
		Let $S$ be a general position set in $D$ and suppose that $S$ is not in general position in $\hat{D}$. Then there are vertices $u,v,w \in V(D)$ and a shortest $u,v$-path $P$ in $\hat{D}$ that passes through $w$. Then the converse $\hat{P}$ of $P$ would be a shortest $v,u$-path in $D$ passing through $w$, a contradiction.
	\end{proof}

	\begin{lemma}\label{lem:diambound}
		The general position number of any digraph $D$ with order $n$ is bounded above by \[ \gp (D) \leq n-\diamm(D)+1.\]
	\end{lemma}
	\begin{proof}
		Let $u_0,u_1,\ldots ,u_{\diamm(D)}$ be a longest geodesic in $D$. Any general position set $S$ can contain at most two vertices from this path, so there are at least $\diam ^*(D)-1$ vertices of the path that do not lie in $S$.
	\end{proof}
	
	\begin{corollary}\label{cor:orientedgpatmostn-2}
		The general position number of a strong oriented graph $D$ with order $n \geq 4$ is at most $n-2$.
	\end{corollary}
	\begin{proof}
		If $D$ has diameter greater than or equal to $3$, then the result follows from Lemma~\ref{lem:diambound}, so we can assume that $\diam (D) = 2$. For $n \geq 2$, a strong transitive digraph is complete, so by Corollary~\ref{cor2.5} we have $\gp (D) \leq n-1$. Suppose that $S = V(D)-\{ x\} $ is a general position set of order $n-1$ of such a digraph. As $n \geq 4$ and $\diam (D) = 2$, there is a vertex $y \in \{ x\} \cup N^+(x)$ with at least two out-neighbours $y_1,y_2$ at distance $d(x,y)+1$ from $x$. As there is at most one arc between $y_1$ and $y_2$, we can assume that $y_1 \not \rightarrow y_2$, and so to avoid having three vertices of $S$ on a shortest $y_1,y_2$-path, the unique shortest $y_1,y_2$-path must be $y_1,x,y_2$. If $y = x$ this is impossible, and if $x \rightarrow y$ this contradicts $d(x,y_2) = 2$. Thus $\gp (D) \leq n-2$.
	\end{proof}
	Corollary~\ref{cor:orientedgpatmostn-2} is tight, since the digraph with vertex set $\{ x_1,\ldots ,x_{n-2}\} \cup \{ y,z\} $ and arcs $z \rightarrow y$ and $y \rightarrow x_i , x_i \rightarrow z$ for $1 \leq i \leq n-2$ has general position number $n-2$.
	
	\begin{theorem}\label{thm:diam}
		If $D$ is a digraph with $\diamm(D) \leq 3$, then $\gp(D) \geq \alpha(D)$.
	\end{theorem}
	\begin{proof}
		Let $D$ have $\diam ^*(D) \leq 3$ and let $A$ be a largest independent set of $D$. If $P$ is a shortest $u,v$-path for $u,v \in A$, then any internal vertex of $P$ is a neighbour of either $u$ or $v$ and hence cannot belong to $A$ by the independence property.
	\end{proof}
	
	Convex subdigraphs can be used to give upper and lower bounds for the general position number of a digraph. Firstly, if $D'$ is a convex subdigraph of $D$, and $u,v \in V(D')$, the class of shortest $u,v$-paths in $D'$ coincides with the class of shortest $u,v$-paths in $D$. It follows that a subset $S \subseteq V(D')$ is in general position in $D'$ if and only if it is in general position in $D$. 
	
	\begin{lemma}\label{lem:convex}
		If $D'$ is a convex subdigraph of a digraph $D$, then $\gp(D') \leq \gp(D)$.
	\end{lemma}   
	
	We now prove a directed version of the `Isometric Cover Lemma' from~\cite{ref9}. A set of isometric subdigraphs $\{ D_1,D_2,\ldots,D_k\}$ of a digraph $D$ is an \emph{isometric cover} of $D$ if $\cup_{i=1}^{k} V(D_i) = V(D)$.
	
	\begin{theorem}\label{thm:isometric cover}
		If $\{ D_1,D_2,\dots,D_k\}$ is an isometric cover of a digraph $D$, then $\gp(D) \leq \sum_{i=1}^{k} \gp (D_i)$.
	\end{theorem}
	\begin{proof}
		Let $S$ be a $\gp$-set of $D$ and set $S_i= S \cap V(D_i)$ for 
		$i\in [k]$. If $S_i$ is not in general position in the isometric subdigraph $D_i$, then it would not be in general position in $D$. As any subset of $S$ is in general position in $D$, it follows that $S_i$ is in general position in $D_i$. Hence $\gp (D_i) \geq |S_i|$ for $i \in [k]$, and \[ |S| \leq \sum _{i = 1}^k |S_i| \leq \sum _{i = 1}^k\gp (D_i).\]     
	\end{proof}
	
	We now describe the effect of adding a source or sink vertex to the digraph; this will be useful when we discuss tournaments.
	
	\begin{lemma}\label{lem:addingsinks}
		The general position of the join of two digraphs $D_1$ and $D_2$ is given by \[ \gp (D_1 \vee D_2) = \gp (D_1)+\gp (D_2).\] 
	\end{lemma}
	\begin{proof}
		Both $D_1$ and $D_2$ are convex subdigraphs of $D_1 \vee D_2$, so $\gp (D_1 \vee D_2) \leq \gp (D_1)+\gp (D_2)$ by Theorem~\ref{thm:isometric cover}. Let $S_i$ be a largest general position set of $D_i$ for $i = 1,2$. Then $S_1 \cup S_2$ is a general position set of $D_1 \vee D_2$, since each vertex of $D_2$ is at distance 1 from $D_1$ and no vertex of $D_1$ is reachable from any vertex of $D_2$.
	\end{proof}
	In particular, the digraph formed from $D$ by adding a source vertex adjacent to every vertex of $D$ or a sink adjacent from every vertex of $D$ has general position number $\gp (D)+1$. More generally, it follows from Theorem~\ref{thm:isometric cover} that adding a source or a sink $x$ to a digraph $D$ gives a digraph $D'$ with general position number $\gp (D)$ or $\gp (D)+1$. We will have $\gp (D')=\gp(D)+1$ if $x$ is a source adjacent only to sinks of $D$ or is a sink adjacent only from sources of $D$.
	
	The structure of general position sets in undirected graphs was characterised in~\cite{AnaChaChaKlaTho}. We now generalise this characterisation to directed graphs.
	
	\begin{theorem}\label{thm:characterisation}
		Let $S$ be a vertex subset of the digraph $D$ and denote the strong components of $\langle S \rangle $ by $S_1,\ldots,S_r$. Then $S$ is a general position set of $D$ if and only if 
		\begin{itemize}
			\item each strong component $S_i$ is a complete digraph, 
			\item $\{ S_1,\dots ,S_r\} $ is \emph{distance-constant}, i.e. for any $1 \leq i < j \leq r$, we have $d(u,v) = d(u',v')$ for any $u,u' \in S_i$ and $v,v' \in S_j$ (where we allow this distance to be infinite),
			\item for no $1 \leq i,j,k \leq r$ is it the case that \[ d(S_i,S_k) = d(S_i,S_j)+d(S_j,S_k) < \infty .\]
		\end{itemize} 
	\end{theorem}
	\begin{proof}
		First we show the necessity of the three conditions. Let $D$ be a digraph and $S$ be a general position set of $D$. By Proposition~\ref{prop:gpsetsaretransitive} each strong component of $\langle S\rangle$ is transitive, and any strongly connected transitive digraph is complete. Now let $u,u' \in S_i$ and $v,v' \in S_j$, where $i \neq j$. As $\langle S_j\rangle$ is a complete digraph, the distance $d(u,v)$ is finite if and only if $d(u,v')$ is finite. Suppose that $d(u,v) < d(u,v')$. As there is an arc $v \rightarrow v'$ in $\langle S_j\rangle$, we must have $d(u,v') = d(u,v)+1$ and there is a shortest $u,v'$-path passing through $v$, a contradiction. Therefore $u$ has the same distance to all vertices of $S_j$, and a similar argument shows that each vertex of $S_j$ is at the same distance from each vertex of $S_i$.  Finally, suppose that $d(S_i,S_k) = d(S_i,S_j)+d(S_j,S_k) < \infty $ for some $i,j,k$. Then as $\{ S_1,\dots ,S_r\} $ is distance-constant there are $u \in S_i$, $v \in S_j$ and $w \in S_k$ such that $d(u,v) = d(S_i,S_j)$, $d(v,w) = d(S_j,S_k)$ and $d(u,w) = d(S_i,S_k)$ and it follows that $D$ would contain a shortest $u,w$-path that passes through $v$, a contradiction.
		
		Now we prove sufficiency of the three conditions. Let $S \subseteq V(D)$ satisfy the three conditions and suppose for a contradiction that there are $u,v,w \in S$ such that $D$ contains a shortest $u,w$-path passing through $v$. As each strong component of $\langle S\rangle$ is complete, $u$ and $w$ do not lie in the same strong component, for otherwise they would be at distance 1 from each other. If we have $u,v \in S_i$ and $w \in S_j$ or $u \in S_i$ and $v,w \in S_j$ for some $i \neq j$, then this would contradict the distance-constant property. Finally, if $u,v$ and $w$ all lie in different strong components, then this would contradict the third property.
	\end{proof}
	
	When applied to an undirected graph, Theorem~\ref{thm:characterisation} reduces to the characterisation result in~\cite{AnaChaChaKlaTho}. In an oriented graph each strong component of $\langle S\rangle$ is a single vertex, and so $S$ induces an acyclic digraph. The \emph{acyclic number} $\overrightarrow{a}(D)$ of a digraph is the largest number of vertices in a subset of $V(D)$ that induces an acyclic subdigraph. 
	
	\begin{corollary}
		For any oriented graph $D$, $\gp (D) \leq \overrightarrow{a}(D)$.
	\end{corollary}
	The paper~\cite{ChaDiSreeThoTui2024+} defines the $\gp $-chromatic number $\chi _{\gp }(G)$ of a graph $G$ to be the smallest number of colours needed to colour the vertices of $G$ such that each colour class is in general position. Extending the definition in the obvious way to digraphs, it follows that $\chi _{\gp }(D) \leq \overrightarrow{\chi } (D)$ for oriented graphs, where $\overrightarrow{\chi } (D)$ is the well-known \emph{dichromatic number} of $D$, the smallest number of colours needed to colour $V(D)$ such that each colour class is acyclic. 
	
	Finally, we consider the complexity of the problem. Since the general position problem for digraphs includes the general position problem for undirected graphs as a special case, and the latter is already known to be NP-complete~\cite{ref9}, it follows trivially that the general position problem for digraphs is NP-complete. A more interesting result is that the general position problem is NP-complete for the class of oriented graphs. We prove that its decision version is NP-complete by a reduction from the Independent Set problem, that, given a graph $G$ and a positive integer $k\leq |V(G)|$, asks for an independent set of vertices having size at least $k$. We denote this problem by \IS, whereas we call \GP the decision version of the general position problem for oriented graphs, that, given an oriented graph $D$ and a positive integer $k'\leq |V(D)|$ asks for a general position set in $D$ of cardinality at least $k'$. We observe that an independent set of a graph $G$ remains an independent set in any orientation of the edges of $G$. 
	
	\begin{figure}
		\centering
		\begin{tikzpicture}[very thick,vertex/.style={circle,draw,minimum size=10,inner sep=1pt,fill=white,font=\scriptsize}]
			\node at (1,0.6) [vertex,fill=red] (s) {{\normalsize $s$}};
			\node at (1,2) [vertex,fill=white] (sb) {$s^2$};
			\node at (0,3) [vertex,fill=red] (sa) {$s^1$};
			\node at (2,3) [vertex,fill=red] (sc) {$s^3$};
			
			\node at (4,0.6) [vertex,fill=white] (t) {{\small $t$}};
			\node at (4,2) [vertex,fill=white] (tb) {$t^2$};
			\node at (3,3) [vertex,fill=red] (ta) {$t^1$};
			\node at (5,3) [vertex,fill=red] (tc) {$t^3$};
			
			\node at (7,0.6) [vertex,fill=white] (u) {{\small $u$}};
			\node at (7,2) [vertex,fill=white] (ub) {$u^2$};
			\node at (6,3) [vertex,fill=red] (ua) {$u^1$};
			\node at (8,3) [vertex,fill=red] (uc) {$u^3$};
			
			\node at (10,0.6) [vertex,fill=white] (v) {{\small $v$}};
			\node at (10,2) [vertex,fill=white] (vb) {$v^2$};
			\node at (9,3) [vertex,fill=red] (va) {$v^1$};
			\node at (11,3) [vertex,fill=red] (vc) {$v^3$};
			
			\node at (13,0.6) [vertex,fill=white] (w) {{\small $w$}};
			\node at (13,2) [vertex,fill=white] (wb) {$w^2$};
			\node at (12,3) [vertex,fill=red] (wa) {$w^1$};
			\node at (14,3) [vertex,fill=red] (wc) {$w^3$};
			
			\node at (7,0.6) [vertex,fill=red] (u) {{\small $u$}};
			
			\node at (13,0.6) [vertex,fill=red] (w) {{\small $w$}};
			
			\path
			(sa) edge[->] (sb)
			(sb) edge[->] (sc)
			(sb) edge[->] (s)
			(sb) edge[->] (u)
			
			(ta) edge[->] (tb)
			(tb) edge[->] (tc)
			(tb) edge[->] (t)
			
			(va) edge[->] (vb)
			(vb) edge[->] (vc)
			(vb) edge[->] (v)
			
			(ua) edge[->] (ub)
			(ub) edge[->] (uc)
			(ub) edge[->] (u)
			
			(wa) edge[->] (wb)
			(wb) edge[->] (wc)
			(wb) edge[->] (w)
			
			(wb) edge[->] (u)
			(s) edge[->] (t) 
			(t) edge[->] (u) 
			(v) edge[->] (u)
			(w) edge[->] (v)

			;
		\end{tikzpicture}
		~\\~\\
		\begin{tikzpicture}[very thick,vertex/.style={circle,draw,minimum size=10,inner sep=1pt,fill=white,font=\scriptsize}]
			\node at (1,0) [vertex,fill=white] (s) {{\normalsize $s$}};
			\node at (1,2) [vertex,fill=white] (sb) {$s^2$};
			\node at (0,3) [vertex,fill=red] (sa) {$s^1$};
			\node at (2,3) [vertex,fill=red] (sc) {$s^3$};
			
			\node at (4,0) [vertex,fill=red] (t) {{\small $t$}};
			\node at (4,2) [vertex,fill=white] (tb) {$t^2$};
			\node at (3,3) [vertex,fill=red] (ta) {$t^1$};
			\node at (5,3) [vertex,fill=red] (tc) {$t^3$};
			
			\node at (10,0) [vertex,fill=red] (v) {{\small $v$}};
			\node at (10,2) [vertex,fill=white] (vb) {$v^2$};
			\node at (9,3) [vertex,fill=red] (va) {$v^1$};
			\node at (11,3) [vertex,fill=red] (vc) {$v^3$};
			
			\node at (7,0) [vertex,fill=white] (u) {{\small $u$}};
			\node at (7,2) [vertex,fill=white] (ub) {$u^2$};
			\node at (6,3) [vertex,fill=red] (ua) {$u^1$};
			\node at (8,3) [vertex,fill=red] (uc) {$u^3$};
			
			\path
			(sa) edge[->] (sb)
			(sb) edge[->] (sc)
			(sb) edge[->] (s)
			(sb) edge[->] (u)
			(sb) edge[->] (v)
			
			(ta) edge[->] (tb)
			(tb) edge[->] (tc)
			(tb) edge[->] (t)
			(tb) edge[->] (v)
			(tb) edge[->] (s)
			
			(va) edge[->] (vb)
			(vb) edge[->] (vc)
			(vb) edge[->] (v)
			(vb) edge[->] (u)
			(vb) edge[->] (t)
			
			(ua) edge[->] (ub)
			(ub) edge[->] (uc)
			(ub) edge[->] (u)
			(ub) edge[->] (s)
			(ub) edge[->] (t)
			
			(s) edge[->] (t) 
			(t) edge[->] (u) 
			(u) edge[->] (v)
			(v) edge[->][bend left=10](s)

			;
		\end{tikzpicture}
		
		\caption{Two examples of $\Phi (D)$ from the proof of Theorem~\ref{thm:NP} with optimal general position sets in red, applied to an oriented path (above) and an oriented cycle (below).}
		\label{fig:NP}
	\end{figure}
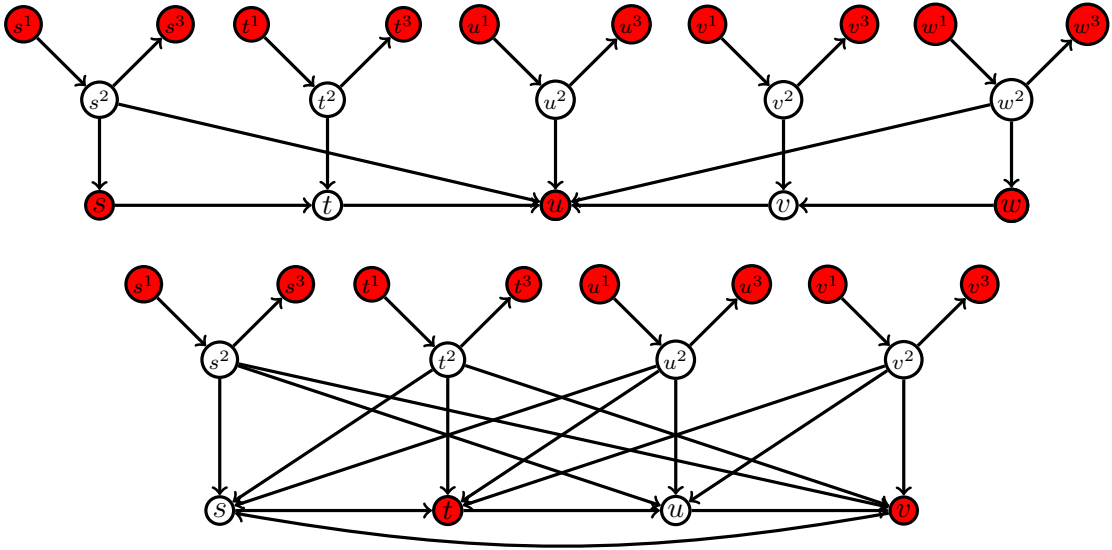

	\begin{theorem}\label{thm:NP}
		\GP is NP-complete even for oriented graphs $D$ with \newline $\diamm(D)=3$.
	\end{theorem}
	\begin{proof}
		We note that \GP is in NP since, given a subset $S$ of $k'$ vertices of the input graph, it is possible to check in polynomial time if $S$ is in general position. Let $I = (G,k)$ be an instance of \IS. Note that \IS is the same problem as the {\sc Clique} problem in the complement graph, which is NP-complete~\cite{Karp72}. We construct an instance $I' = (D,k')$ of \GP such that $I$ has a yes answer if and only if $I'$ has a yes answer.
		
		For a digraph $D$ with order $n$ and no isolated vertices we define the digraph $\Phi (D)$ as follows. We add to $D$ a set $V^{(1)} \cup V^{(2)} \cup V^{(3)}$ of vertices, where $V^{(i)} = \{ v^i : v \in V(D)\} $ for $i = 1,2,3$, as well as the arcs $v^1 \rightarrow v^2$, $v^2 \rightarrow v^3$ and $v^2 \rightarrow v$ for all $v \in V(D)$. Finally for each pair $u,v \in V(D)$ we add the arc $u^2 \rightarrow v$ if and only if $u$ reaches $v$ in $D$ and $d_D(u,v) \geq 2$. As the vertices $u^1,u^3$ are the only leaves in $\Phi (D)$, it is simple to recover $D$ from $\Phi (D)$. See Figure~\ref{fig:NP} for two examples of the operation $\Phi $ together with largest general position sets of the resulting digraphs. Observe that $\diamm (\Phi (D)) = \max \{ 3,\diamm (D)\} $.
		
		Let $G$ be any connected graph with order $n$ and set $k' = 2n+4+k$. Let $G' = G \vee K_2$, where $V(K_2) = \{ x,y\} $. Note that $\alpha (G') = \alpha (G)$. Now form a digraph $D'$ from $G'$ with arcs $x \rightarrow y$ and  $u \rightarrow x$, $y \rightarrow u$ for all $u \in V(G')$, and orient the edges in $E(G)$ at random. We have $\diam (D') \leq 3$, so $\diamm (\Phi (D')) = 3$. 
		
		We now show that $\gp (\Phi (D')) = 2(n+2)+\alpha (D') = 2n+4+\alpha (G)$. Note that if $A$ is a largest independent set of $G$ (and hence $A$ is also a largest independent set of $D'$), then $V^{(1)} \cup V^{(3)} \cup A$ is a largest independent set of $\Phi (D')$, and so by Theorem~\ref{thm:diam} we have $\gp (\Phi (D')) \geq 2(n+2)+\alpha (D')$. For the converse, let $S$ be a largest general position set of $\Phi (D')$. Suppose that $v^2 \in S$ for some vertex $v \in V(D')$. Then as $v^1,v^2,v^3$ is a shortest path in $\Phi (D')$, we cannot have both $v^1,v^3 \in S$. Thus we can replace $v^2$ by a vertex in $\{ v^1,v^3\} - S$ to get a general position set of the same order. By repeating this operation, we see that we can assume that $S \cap V^{(2)} = \emptyset $. Now suppose that $u,v \in S$ for some $u,v \in V(D')$, where $u \rightarrow v$. As $u^1,u^2,u,v$ is a shortest path in $\Phi (D')$, we see that $u^1 \not \in S$. Thus we can replace $u$ in $S$ by $u^1$ and still have a general position set. By repeating this, we can now assume that $S \cap V(D')$ to be an independent set of $D'$. Hence $|S| \leq 2(n+2)+\alpha (D')$ and the equality is established.
		
		Hence $I$, an instance of \IS for $G$, has a positive answer if and only if $I'$ has a positive answer for $\Phi (D')$.
	\end{proof}

	\section{Families of digraphs}\label{sec:families}
	
	In this section we explore the general position number for certain families of digraphs. 
	
	\subsection{Circulant digraphs}
	Recall that the \emph{circulant digraph} $\Circ(n,C)$ has vertex set $\Z_n=\{0,1,\ldots,n-1\}$ and a directed edge from $x$ to $x+c\pmod{n}$ for every $x\in\Z_n$ and every $c\in C$. We deal here with the special case where $C=\{1,2,\ldots,d\}$ for some fixed $d\geq 1$ and $d < n$, so that $\Circ(n,\{1,2,\ldots,d\})$ is a $d$-regular digraph of order $n$. In the rest of this subsection we will assume that $d \geq 1$ and write $n = qd+r$, where $q \geq 0$ and $2 \leq r \leq d+1$. 
	
	\begin{lemma}\label{lem:gpcirc}
		The subset $S = \{ 0,\ldots ,r-1\} $ is in general position in $\Circ(n,\{1,2,\ldots,d\})$.
	\end{lemma}
	\begin{proof}
		Suppose that there is a shortest path $P$ in $\Circ(n,\{1,2,\ldots,d\})$ containing three distinct vertices $i,j,k \in S = \{ 0,\ldots ,r-1\} $, where $i < j < k$. Observe that $d(i,j) = d(i,k) = d(j,k) = 1$, the diameter is $\diam ( \Circ(n,\{1,2,\ldots,d\}) ) = q+1$ and $d(j,i) = d(k,i) = d(k,j) = q+1$. Hence $P$ must be a shortest $k,i$-path through $j$, a shortest $k,j$-path through $i$ or a shortest $j,i$-path through $k$, but by the foregoing observation each of these paths would have length $q+2$, a contradiction. Thus $S$ is in general position.
	\end{proof}
	
	\begin{theorem}\label{thm:dircirc}
		The general position number of $\Circ(n,\{1,2,\ldots,d\})$ satisfies
		\[r\leq \gp(\Circ(n,\{1,2,\ldots,d\}))\leq d+1.\] When $n\equiv 1\pmod{d}$, $\gp(\Circ(n,\{1,2,\ldots,d\})) = d+1$.
	\end{theorem}
	\begin{proof}
		The lower bound comes from Lemma~\ref{lem:gpcirc}. For the upper bound, let $S'$ be a largest $\gp $-set. If no two elements of $S'$ are congruent $\pmod{d}$, then $|S'| \leq d$. Thus we can assume that $S'$ contains distinct vertices $i,j$, where $0 \leq i < j \leq n-1$, such that $i \equiv j \pmod{d}$. Denote by $Z$ the cycle subdigraph of $\gp(\Circ(n,\{1,2,\ldots,d\}))$ obtained by including only arcs associated with the generator $1$. Choose a congruent pair $i,j$ so that the distance $d_Z(i,j)$ in $Z$ is as small as possible. Then by symmetry we can assume that $i = 0$ and $j = j'd$ for some $j'$. Suppose that $k \in S' \cap [j+1,n-1]$, where $k = q'd+t, 0 \leq t \leq d-1$. Then a shortest $0,k$-path is given by $0,d,\ldots ,j'd,\ldots ,q'd$, followed by the arc $q'd \rightarrow q'd+t$ if $t \neq 0$, which would contain three vertices of $S'$. Therefore $S' \subseteq [0,j]$. By the minimality of $d_Z(i,j)$, no vertices of $S'-[0,j]$ are congruent $\pmod{d}$ and none are congruent to $0 \pmod{d}$. Therefore $|S'| \leq d+1$, establishing the upper bound. The upper and lower bounds coincide when $n\equiv 1\pmod{d}$, implying the final equality.
	\end{proof}
	For many values of $n$ and $d$, we can give a much stronger lower bound.
	
	\begin{theorem}\label{thm:ncong2}
		If $n \equiv a \pmod{d}$ where $a\geq 2$, $\gp(\Circ(n,\{1,2,\ldots,d\})) \geq \left \lceil \frac{d+1}{a} \right \rceil $.
	\end{theorem}
	\begin{proof}
		Let $n = qd + a$. Suppose first that $d$ is a multiple of $a$ and set $S = \{ 0,a,2a,\ldots,d-a,d\} $. If $i,j \in S$, where $j > i$, then $d(i,j) = 1$ and $d(j,i) = q$. It is now easy to check that $S$ is in general position as in the proof of Lemma~\ref{lem:gpcirc}. The cases of other congruence classes of $d$ mod $a$ are similar; if $d\equiv b\pmod{a}$ then the relevant set is $S = \{ 0,a,2a,\ldots,d-b\}$.
	\end{proof}
	Computation shows that the larger of the sets exhibited in Lemma~\ref{lem:gpcirc} and Theorem~\ref{thm:ncong2} (when $n \equiv a \pmod{d}, a\geq 2$) is optimal for $n \leq 50$ and $d \leq 15$, and we conjecture that this holds for all $n$ and $d$.

	\subsection{Kautz digraphs}
	Let $m\geq 3$ and $k\geq 2$. The \emph{Kautz digraph} $\Ka(m,k)$ has vertex set consisting of all words of length $k$ over an alphabet of $m$ characters, with the property that no two consecutive characters in the word are equal. Thus the order of the graph is $m(m-1)^{k-1}$. There is a directed edge from a word $a_1 a_2\ldots a_k$ to every word of the form $a_2\ldots a_k b$ where $b\neq a_k$. Thus the graph has degree $m-1$. The Kautz digraphs are of substantial interest in the degree-diameter problem, having the smallest possible diameter amongst all digraphs of their order and degree. In the case $k=2$, the Kautz digraph $\Ka(m,2)$ can be viewed as the line digraph of the complete (undirected) graph $K_m$.
	
	We use the alphabet $1,2,\ldots 9,a,b,\ldots$ in our examples. Figure~\ref{fig:Ka52} shows the Kautz digraph $\Ka(5,2)$ with a maximal gp-set of cardinality 8 highlighted. In Table~\ref{tab:kautz} we show the gp numbers of $\Ka(m,k)$ for $k = 3,4$ and some small values of $m$, determined by computer. 
	
	\begin{theorem}\label{thm:Kautz}
		For $m \geq 3$, $\gp (\Ka(m,2)) = \left \lfloor \frac{m^2}{3} \right \rfloor $.
	\end{theorem}
	\begin{proof}
		First we give a construction of the claimed size. Divide the alphabet of $m$ characters into three parts $A_1,A_2,A_3$ differing in size by at most 1 and form a digraph $K(m)$ by adding every arc $x \rightarrow y$, where $x \in A_i$, $y \in A_j$, $i < j$ (observe that this is an orientation of the tripartite Tur\'{a}n graph). Let $S \subset V(\Ka (m,2))$ consist of each word $xy$ where $x \rightarrow y$ in $K(m)$. The graph $\Ka (m,2)$ has diameter 2. Suppose that $P$ is a path of length 2 with all vertices in $S$. If the initial vertex of $P$ is of the form $xy$, $x \in A_1, y \in A_2$, then the next vertex of $P$ begins with the letter $y$, and to lie in $S$ this vertex must have the form $yz$, where $z \in A_3$. However, in this case the third vertex of $P$ begins with the letter $z$, and there are no words in $S$ of this form. Similarly, $P$ cannot begin with a word $yz$ or $xz$, where $x \in A_1, y \in A_2, z \in A_3$, since the next vertex of the path would start with the letter $z$. 
		
		Now we show that this construction is optimal. A computer check showed that the claimed result is true for $m =3,4$. Suppose that the result is true for $\Ka (m',2)$ for $m' < m$ and let $S'$ be a largest general position set of $\Ka (m,2)$. Suppose that $S'$ contains two words with the same letters, for example $12$ and $21$. Now $S'$ cannot have any other vertices containing the letters $1$ or $2$; for instance, if $31 \in S$, then $31 \rightarrow 12 \rightarrow 21$ would be a shortest path through three vertices of $S'$. Therefore the remaining $\gp (\Ka (m,2))-2$ vertices of $S'$ lie in the isometric subdigraph of $\Ka (m,2)$ induced by the words not containing the letters $1$ and $2$. This subdigraph is isomorphic to $\Ka (m-2,2)$, and we conclude by induction that we would have $|S'| \leq 2+\left \lfloor \frac{(m-2)^2}{3} \right \rfloor < \left \lfloor \frac{m^2}{3} \right \rfloor$. 
		
		Now we associate with $S'$ the \emph{auxiliary digraph} $\Omega ^{\rightarrow }(S')$ with vertex set $[1,m]$ and an arc $i \rightarrow j$ when $ij \in S'$. The \emph{auxiliary graph} $\Omega (S')$ is the underlying undirected graph of $\Omega ^{\rightarrow }(S')$. By the preceding argument, $\Omega ^{\rightarrow }(S')$ and $\Omega (S')$ have the same size. Suppose that $|S'| > \left \lfloor \frac{m^2}{3} \right \rfloor $. Then by Tur\'{a}n's Theorem the graph $\Omega (S')$ must contain a copy of $K_4$. Denote the vertices of this $K_4$ by $\{ w,x,y,z\} $. In $\Omega (S')$ the copy of $K_4$ cannot contain a directed triangle $x \rightarrow y \rightarrow z \rightarrow x$, for otherwise $xy,yz,zx$ would be a shortest path containing three vertices of $S'$, and cannot contain a directed path of length 3, say $x \rightarrow y \rightarrow w \rightarrow z$, for this corresponds to the shortest path $xy,yw,wz$ in $\Ka (m,2)$. Suppose that there is a source (say it is $x$) in $\Omega ^{\rightarrow }(S')$. Then we can assume that $y \rightarrow w$ in $\Omega (S')$. Now it is easy to see that the remaining edges cannot be oriented without creating a directed triangle or path of length 3. For the same reason there can be no sink. Hence we can assume that $x \rightarrow y$, $x \rightarrow z$ and $w \rightarrow x$. Now the only way to avoid directed triangles and paths of length 3 is to make $z$ a sink. Hence $|S'| \leq \left \lfloor \frac{m^2}{3} \right \rfloor $ and the proof is complete. 
	\end{proof}
	
	For $k=3$, over the range calculated we have that $\gp(\Ka(m,2))$ is the sum of squares up to $k-1$ (see \url{https://oeis.org/A000330}). 
	
	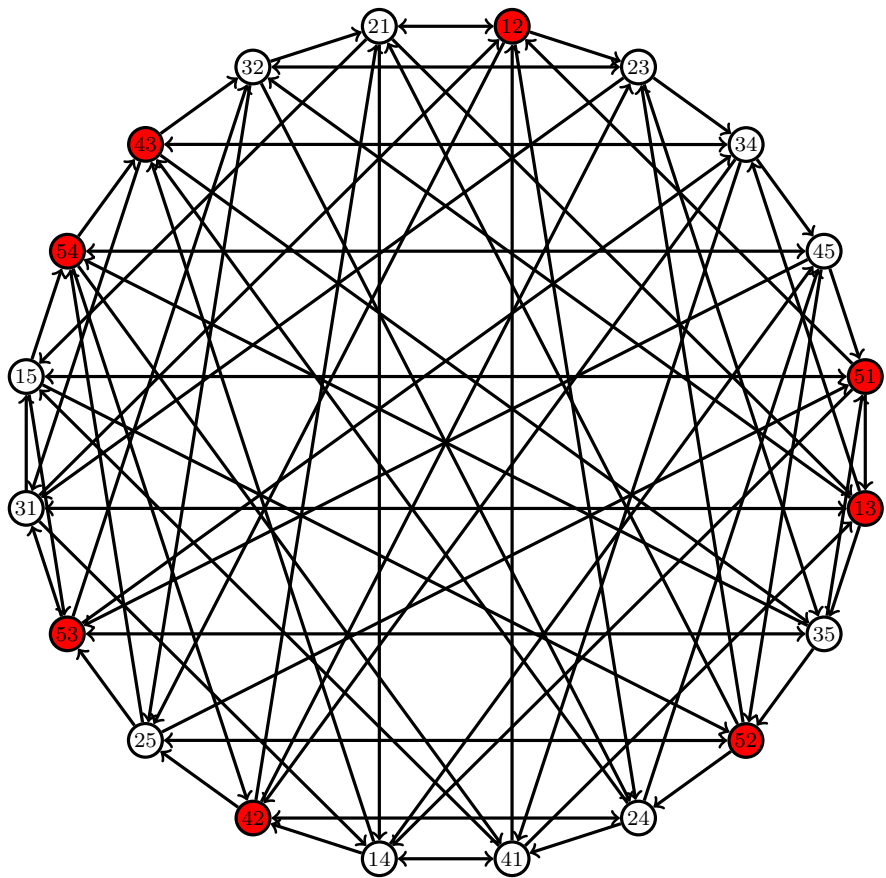
\begin{figure}
		\centering
		\input{Ka52}
		\caption{The Kautz digraph $\Ka(5,2)$ and a maximal gp-set}
		\label{fig:Ka52}
	\end{figure}
	
	\begin{table}
		\centering
		\begin{tabular}{|cccp{12cm}|}
			\hline
			$k$ & $m$ & gp & Example gp-set \\
			\hline
			$3$ & $3$ & $5$ &$121, 123, 313, 321, 323$ \\
			& $4$ & $14$ & $121, 124, 131, 132, 134, 231, 232, 234, 414, 421, 424, 431$\par
			$432, 434$ \\
			& $5$ & $30$ & $121, 124, 131, 132, 134, 135, 151, 152, 154, 231, 232, 234$\par
			$235, 251, 252, 254, 414, 421, 424, 431, 432, 434, 435, 451$\par
			$452, 454, 531, 532, 534, 535$ \\
			& $6$ & $55$ & $212, 213, 214, 215, 216, 232, 242, 243, 246, 252, 253, 254$\par
			$256, 262, 263, 312, 313, 314, 315, 316, 342, 343, 346, 352$\par
			$353, 354, 356, 362, 363, 412, 413, 414, 415, 416, 452, 453$\par
			$454, 456, 512, 513, 514, 515, 516, 612, 613, 614, 615, 616$\par
			$642, 643, 646, 652, 653, 654, 656$ \\
			& $7$ & $91$ & $212, 213, 214, 215, 216, 217, 252, 253, 254, 262, 263, 264$\par
			$265, 272, 273, 274, 275, 276, 312, 313, 314, 315, 316, 317$\par
			$323, 324, 343, 352, 353, 354, 362, 363, 364, 365, 372, 373$\par
			$374, 375, 376, 412, 413, 414, 415, 416, 417, 423, 424, 452$\par
			$453, 454, 462, 463, 464, 465, 472, 473, 474, 475, 476, 512$\par
			$513, 514, 515, 516, 517, 562, 563, 564, 565, 572, 573, 574$\par
			$575, 576, 612, 613, 614, 615, 616, 617, 672, 673, 674, 675$\par
			$676, 712, 713, 714, 715, 716, 717$ \\
			\hline
			$4$ & $3$ & $7$ & $1312, 1313, 2121, 2132, 2312, 2313, 2321$ \\
			& $4$ & $27$ & $1312, 2124, 2132, 2134, 2141, 2142, 2143, 2312, 2324, 2341, 2342, 2343$\par
			$2412, 2424, 2432, 3132, 3134, 3141, 3142, 3143, 4132, 4134, 4141, 4142$\par
			$4312, 4342, 4343$ \\
			\hline
		\end{tabular}
		\caption{gp numbers of Kautz digraphs $\Ka(m,k)$}
		\label{tab:kautz}
	\end{table}
	
	\subsection{Permutation digraphs}
	
	Let $d\geq 2$ and $k\geq 2$. The \emph{permutation digraph} $\Pe(d,k)$ has a vertex set consisting of all words of length $k$ chosen from the alphabet of $\{ 0,1,\ldots ,d+k-1\} $, where no letter is used more than once in a word. A word $(x_1,x_2,\ldots,x_k)$ has an arc to any word of the form $(x_2,\ldots,x_k,y)$ where $y\notin\{x_1,\ldots,x_k\}$. Thus the graph is diregular of degree $d$, and has $(d+k)_k=(d+k)(d+k-1)\cdots(d+1)$ vertices. These digraphs are of interest because they have a unique path property: no two vertices have more than one directed path of length at most $k$ between them (we say the graph is $k$-\emph{geodetic}).
	
	\begin{theorem}
		For $d \geq 2$, $\gp (\Pe(d,2)) = \left \lceil \frac{(d+2)^2}{4} \right \rceil $.
	\end{theorem}
	\begin{proof}
		Partition the alphabet $[0,d+1]$ into two parts $X_1,X_2$ with order differing by at most 1. Form a set $S$ by including the word $ij$ if and only if $i \in X_1$ and $j \in X_2$. This contains $\left \lceil \frac{(d+2)^2}{4} \right \rceil $ words. Observe that $S$ is an independent set, and so we only need to check the no-three-in-line property for two vertices at distance $\geq 4$ apart, but the vertices of $S$ are all at distance $\leq 3$. Hence $\gp (\Pe(d,2)) \geq \left \lceil \frac{(d+2)^2}{4} \right \rceil $. 
		
		Now suppose for a contradiction that $\Pe(d,2)$ has a general position set $S'$ with order $> \left \lceil \frac{(d+2)^2}{4} \right \rceil $. We form the auxiliary digraph and graph $\Omega ^{\rightarrow }(S')$ and $\Omega (S')$ in the same manner as the proof of Theorem~\ref{thm:Kautz}, but with vertex set $[0,d+1]$. Suppose that $\Omega (S')$ contains a digon, or, equivalently, $S'$ contains two words on the same two letters, for example $01$ and $10$. The distance between these vertices in $\Pe (d,2)$ is 4, and it is easy to verify that every vertex of $P(d,2)$ lies on either a shortest path from $01$ to $10$ or a shortest path from $10$ to $01$~\cite{Brunat}, so that such a general position set would contain just two vertices. Therefore the size of the graph $\Omega (S')$ is the same as the size of the digraph $\Omega ^{\rightarrow }(S')$. 
		
		As $\Omega (S')$ has size $> \left \lceil \frac{(d+2)^2}{4} \right \rceil $, it follows from~\cite{Dirac,Erdos} that for $d \geq 5$, $\Omega (S')$ contains a diamond, i.e.\ a copy of $K_4^-$. $\Omega ^{\rightarrow }(S')$ cannot contain a directed triangle $i,j,k,i$, for otherwise $S'$ would contain all three vertices of the shortest path $ij,jk,ki$. Up to symmetry there are six remaining orientations of $K_4^-$ to consider, and each of them leads to a similar violation of the no-three-in-line property. It follows that $|S'| \leq \left \lceil \frac{(d+2)^2}{4} \right \rceil$ for $d \geq 5$. The cases $d \leq 4$ are easily checked by computer.

	\end{proof}

	\begin{table}
		\centering
		\begin{tabular}{|cccp{11cm}|}
			\hline
			$k$ & $d$ & gp & Example gp-set \\
			\hline
			$3$ & $2$ & $12$ & $120, 210, 301, 302, 304, 310, 320, 401, 402, 403, 410, 420$ \\
			& $3$ & $24$ & $021, 024, 031, 034, 051, 054, 201, 204, 231, 234, 251, 254$\par
			$301, 304, 321, 324, 351, 354, 501, 504, 521, 524, 531, 534$\\
			& $4$ & $40$ & $501, 502, 503, 504, 510, 512, 513, 514, 520, 521, 523, 524$\par
			$530, 531, 532, 534, 540, 541, 542, 543, 601, 602, 603, 604$\par
			$610, 612, 613, 614, 620, 621, 623, 624, 630, 631, 632, 634$\par
			$640, 641, 642, 643$\\
			\hline
		\end{tabular}
		\caption{$\gp $-numbers of permutation digraphs}
		\label{tab:permdigraph}
	\end{table}
	
	In Table~\ref{tab:permdigraph}, we show the $\gp $-numbers of permutation digraphs for $k = 3$ and $2 \leq d \leq 4$, determined by computer. The alphabet of $d+k$ letters used for the example gp-sets is $0,1,\ldots,d+k-1$. The results suggest that $\gp (\Pe (d,3)) = 2d(d+1)$. We prove a lower bound for $k \geq 3$ and $d \geq 2k$ that includes this as a special case.
	
	\begin{theorem}\label{thm:permutationlowerbound}
		For $k \geq 3$ and $d \geq 2k$, $\gp (\Pe (d,k)) \geq 2(d+k-2)_{k-1}$.
	\end{theorem}
	\begin{proof}
		Take the vertices of $S$ to be $i_1i_2\ldots i_{k-1}\epsilon $, where $\epsilon = 0,1$ and $i_1i_2\ldots i_{k-1}$ is any permutation of length $k-1$ from the alphabet $\{ 0,\ldots ,d+k-1\} -\{ 0,1\} $. This has the stated order. Each pair of vertices in this set is at least distance $k$ apart. It is proven in~\cite{Brunat} that for $d \geq 2k$ the diameter of $\Pe (d,k)$ is $2k$. However, no pair of vertices from our set is at distance $2k$ apart, so we cannot have three vertices of $S$ on a common shortest path.
	\end{proof}
	
	We conjecture that the lower bound from Theorem~\ref{thm:permutationlowerbound} is exact. 
	
	\subsection{Trees}\label{subsec:trees}
	
	Finding the general position of an oriented tree appears to be a difficult problem. As each leaf of a tree is a source or a sink, Lemma~\ref{lem:ext} implies the following bound.  
	
	\begin{lemma}\label{lem:leaf bound}
		For any oriented tree $T^{\rightarrow }$ with leaf number $\ell (T^{\rightarrow })$, the general position number of $T^{\rightarrow }$ is bounded below by $\gp (T^{\rightarrow }) \geq \ell (T^{\rightarrow })$.
	\end{lemma}
	
	We show that Lemma~\ref{lem:sourcesandsinks} gives the exact value of the general position number for arborescences. An \emph{out-arborescence} rooted at $x \in V(T)$ is an oriented tree $T^{\rightarrow }$ such that all arcs are oriented away from $x$. 
	
	\begin{theorem}\label{thm:arborescence}
		If $T^{\rightarrow }$ is an out-arborescence, then $\gp (T^{\rightarrow }) = |\sink (T^{\rightarrow })|+|\nsink (T^{\rightarrow })|$.
	\end{theorem}
	\begin{proof}
		Lemma~\ref{lem:sourcesandsinks} shows that $\sink (T^{\rightarrow }) \cup \nsink (T^{\rightarrow })$ is in general position. Let $S$ be any largest general position set in $T^{\rightarrow }$ and suppose that there is some $y \in S - (\sink (T^{\rightarrow }) \cup \nsink (T^{\rightarrow }))$. Let $y,y_1,\dots ,y_t$ be a longest path in $T^{\rightarrow }$ starting at $y$. Then $y_t$ is a sink and $y_{t-1}$ is a near-sink. We cannot have both $y_{t-1}$ and $y_t$ contained in $S$, for then the above path would contain three vertices of $S$. Hence we can replace $y$ by one of the vertices of $\{ y_{t-1},y_t\} -S$ without creating three-in-a-line. Therefore, in at most $|S|$ steps, $S$ can be transformed into $\sink (T^{\rightarrow }) \cup \nsink (T^{\rightarrow })$, and it follows that the latter is a largest general position set.
	\end{proof}
	The analogous result for in-arborescences, sources and near-sources follows upon taking the converse.  We see from Theorem~\ref{thm:arborescence} that the general position number of a tree $T^{\rightarrow }$ can be arbitrarily larger than $|\ext(T^{\rightarrow })|$. 
	
	Note that the lower bounds on the general position number of a tree $T^{\rightarrow}$ implied by Lemmas~\ref{lem:ext} and~\ref{lem:sourcesandsinks} are not tight in all cases. Figure~\ref{fig:tree7gp5} shows an example of a tree and its converse, where Lemmas~\ref{lem:ext} and~\ref{lem:sourcesandsinks} imply a lower bound of 4 on the gp number, but the actual value is 5. These are the smallest trees with this property.
	
	\begin{figure}
		\centering
		\input{tree7gp5}
		\caption{A tree and its converse with gp number 5}
		\label{fig:tree7gp5}
	\end{figure}
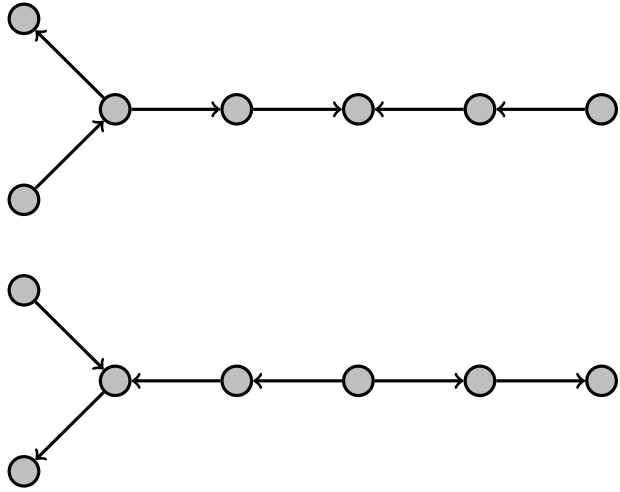

	\section{Orientations of graphs}\label{sec:orientation}

	For any graph $G$ of order $n\geq 2$, there are different orientations of $G$ that have different general position numbers. This leads to the introduction of the following definitions.
	
	\begin{definition}
		For a graph $G$ of order $n \geq 2$, the \emph{upper oriented general position number} $\ugp(G)$ of $G$ is the maximum general position number of an orientation of $G$. Similarly, the \emph{lower oriented general position number} $\lgp(G)$ of $G$ is defined as the minimum general position number of an orientation of $G$. 
	\end{definition}
	
	\begin{definition}
		The \emph{general position spectrum} of a graph $G$ is defined as the set $$ S_{\gp}(G) = \{  \gp(D) : D \text{ is an orientation of } G \}.$$
	\end{definition}
	
	Thus $\lgp(G) = \min  S_{\gp}(G)$ and $ \ugp(G) = \max  S_{\gp}(G)$. Trivially for any graph $G$ of order $n \geq 2$ we have $2 \leq \lgp(G) \leq \ugp(G) \leq n$. 
	
	\subsection{Bounding the oriented general position numbers}
	
	Our main conjecture is that for any graph one can find orientations with different general position numbers, or equivalently the general position spectrum always has cardinality at least 2. A computer check shows that the conjecture is true for graphs with order at most 6.
	
	\begin{conjecture}\label{conj:orientation}
		For any connected graph $G$ with order $n \geq 3$, $\ugp (G) > \lgp (G)$.  
	\end{conjecture}
	
	The upper oriented general position number can be characterised in terms of induced comparability subgraphs. Note that the comparability subgraphs in question may be disconnected.
	
	\begin{theorem}\label{thm:comparability subgraphs}
		The upper oriented general position number of $G$ is equal to the largest order of an induced subgraph $H$ of $G$ that is also a comparability graph. 
	\end{theorem}
	\begin{proof}
		Let $G$ be a graph and $D$ be an orientation of $G$ with largest possible general position number, with $S$ a largest general position set. By Proposition~\ref{prop:gpsetsaretransitive} the subdigraph $\langle S\rangle$ is transitive, and it follows that the subgraph $H$ induced by the vertices corresponding to $S$ is a comparability graph. 
		
		Conversely, let $H$ be an induced comparability subgraph of $G$. Orient the edges of $H$ to give a transitive orientation of $V(H)$. Then orient any edge between $V(H)$ and $V(G-H)$ in the direction from $V(H)$ to $V(G-H)$, and orient all remaining edges at random. It is easily seen that $V(H)$ is a general position set in this orientation.
	\end{proof}
	
	Theorem~\ref{thm:comparability subgraphs} also shows that the difference $n-\ugp (G)$ can be arbitrarily large; as a comparability graph does not contain induced odd cycles of length 5 or more, any graph containing $r$ disjoint induced odd cycles of length $\geq 5$ has $\ugp (G) \leq n-r$. 
	
	\begin{corollary}\label{cor:comparability}
		A graph $G$ satisfies $\ugp (G) = n$ if and only if $G$ is a comparability graph.
	\end{corollary}
	
	It follows that $\ugp (G) = n$ can be tested in polynomial time~\cite{McConnell}. It appears that the order of a largest induced comparability subgraph has not been studied before. However, both cliques and bipartite graphs are comparability graphs, and largest independent unions of cliques~\cite{unionofcliques} and largest induced bipartite subgraphs~\cite{bipartitesubgraphs} have both been been studied. An important subcase of the largest induced bipartite subgraph problem is finding the largest induced forest in a graph, see~\cite{Alon}.  For any graph $G$, we denote the number of vertices in a largest induced union of cliques of $G$ (possible containing just one clique) by $\alpha ^{\omega }(G)$, the number of vertices in a largest induced bipartite subgraph of $G$ by $\alpha _2(G)$ and the number of vertices in a largest induced forest subgraph of $G$ by $\ac(G)$. Trivially we have $\alpha _2(G) \geq \ac (G)$.

	\begin{corollary}\label{cor:induced bipartite bound}
		For any graph $G$, $\ugp (G) \geq \max \{ \alpha ^{\omega }(G),\alpha _2(G)\} $.
	\end{corollary}
	
	For any $n \geq 2$ there is a graph with order $n$ and general position number 2, namely the path $P_n$. By contrast, using Ramsey's Theorem and Corollary~\ref{cor:induced bipartite bound} we see that the smallest value of the general position number taken over all digraphs with order $n$ tends to infinity with $n$. Since any general position set of a graph is an independent union of cliques~\cite{AnaChaChaKlaTho}, it also follows from Corollary~\ref{cor:induced bipartite bound} that $\ugp (G) \geq \gp (G)$ for any graph. Theorem~\ref{thm:comparability subgraphs} also implies a lower bound in terms of the diameter.
	
	\begin{corollary}\label{cor:diam bound}
		For any graph $G$, $\ugp (G) \geq \diam (G)+1$, with equality if and only if $G$ is a path. For any graph with order $n \geq 4$, it holds that $\ugp(G) \geq 4$
	\end{corollary}
	\begin{proof}
		The lower bound $\ugp (G) \geq \diam (G)+1$ follows instantly from Theorem~\ref{thm:comparability subgraphs} applied to a longest geodesic, as does the equality case when $G$ is a path. Suppose that $\ugp (G) = \diam (G)+1$ and $G$ is not a path. Let $P = u_0,u_1,\ldots ,u_{\diam (G)}$ be a longest geodesic in $G$. As $G$ is not a path, there are vertices of $G$ not on $P$. However, any vertex of $G-P$ must be adjacent to a vertex of $P$; otherwise, if $u \in V(G-P)$ has no edge to $P$, then $\{ u\} \cup V(P)$ would be a comparability graph, a contradiction by Theorem~\ref{thm:comparability subgraphs}. Let $j$ be the smallest value of the index $i$ such that $u_i$ has some neighbour $v$ not on $P$. Then the only other neighbours that $v$ can have on $P$ are $u_{j+1}$ and $u_{j+2}$, and it is easily verified that each of the four corresponding induced subgraphs of $G$ are comparability graphs, which would again contradict Theorem~\ref{thm:comparability subgraphs}. 
		
		By Theorem~\ref{thm:comparability subgraphs} complete graphs and paths with order $n \geq 4$ have upper oriented general position number $n$, and $\ugp (G) \geq 4$ follows for any other graph with diameter at least 2 by the preceding argument.
	\end{proof}
	
	The following bounds on $\ac (G)$ are given in~\cite{Alon}, all of which give lower bounds for $\ugp (G)$. We will use one of these bounds later to prove Conjecture~\ref{conj:orientation} for triangle-free subcubic graphs.
	
	\begin{theorem}{\rm \cite{Alon}}\label{thm:Alon}
		\text{ }
		\begin{itemize}
			\item If $G$ is a triangle-free graph with maximum degree 3, then $\ac (G) \geq \frac{5n}{8}$.
			\item If $G$ is triangle-free and has size $m$, then $\ac (G) \geq n-\frac{m}{4}$.
			\item If $G$ is a connected graph with maximum degree $\Delta \geq 3$, then $\ac (G) \geq \alpha (G) +\frac{n-\alpha (G)}{(\Delta -1)^2}$.
			
		\end{itemize}
	\end{theorem}
	
	Now we consider an upper bound for the lower oriented general position number. The \emph{path cover number} $\rho (G)$ of $G$ is the smallest number of vertex-disjoint paths in $G$ that cover every vertex of $G$. 
	
	\begin{lemma}\label{lem:path cover bound}
		For any graph $G$, if $V(G)$ can be partitioned into $r$ paths of length at least 1 and $s$ copies of $K_1$, then $\lgp (G) \leq 2r+s$. In particular, $\lgp (G) \leq 2\rho (G)$.
	\end{lemma}
	\begin{proof}
		Let $\{ P^{(1)},P^{(2)},\ldots ,P^{(t)}\} $ be a collection of $t$ vertex-disjoint paths of $G$ that covers $V(G)$, where $r$ of the paths have length at least 1 and $s$ of them have length 0 (note, the superscripts only label the paths and do not indicate their lengths). For $1 \leq i \leq t$ we orient the edges of the induced subgraph $G[P^{(i)}]$ as follows. Suppose that $P^{(i)}$ has length $\ell (i) \geq 0$ and write the vertices of $V(P^{(i)})$ as $\{ v_0^i,\ldots ,v_{\ell (i)}^i\} $ so that the path $P^{(i)}$ is $v_0^i,\dots ,v_{\ell (i)}^i$. Orient the edges of $P^{(i)}$ so that $v_j^i \rightarrow v_{j+1}^i$ for $0 \leq j \leq \ell (i)-1$. In general $P^{(i)}$ will not be induced, and if there is a chord $v_{j_1}^iv_{j_2}^i$ in $G$ with $j_1 +2 \leq j_2$ we orient it as $v_{j_2}^i \rightarrow v_{j_1}^i$. 
		
		Finally, if $e$ is an edge of $G$ between $P^{(i)}$ and $P^{(j)}$, where $1 \leq i < j \leq t$, orient $e$ so that its initial vertex is in $P^{(i)}$ and its terminal vertex is in $P^{(j)}$. Then the subdigraph induced by $V(P^{(i)})$ is isometric for $1 \leq i \leq t$ and has general position number at most 2 by Lemma~\ref{lem:diambound} (and is just 1 if the path $P^{(i)}$ has length 0), and it follows from Theorem~\ref{thm:isometric cover} that this orientation of $G$ has general position number at most $2r+s$.
	\end{proof}
	
	\begin{corollary}\label{cor:true for long paths}
		Conjecture~\ref{conj:orientation} is true for graphs containing a path of length at least $n-2$.
	\end{corollary}
	\begin{proof}
		The conjecture is easily seen to be true for $n = 3$, so assume that $n \geq 4$. If $G$ is traceable, then by Lemma~\ref{lem:path cover bound}, $\lgp (G) = 2$, whilst if $G$ has a path of length $n-2$, then Lemma~\ref{lem:path cover bound} gives $\lgp (G) \leq 3$. On the other hand, by Corollary~\ref{cor:diam bound} we have $\ugp (G) \geq 4$. 
	\end{proof}
	
	\begin{corollary}
		Any graph with order $n \geq 3$ and size $m \geq \binom{n-3}{2}+3$ satisfies Conjecture~\ref{conj:orientation}.
	\end{corollary}
	\begin{proof}
		It is shown in~\cite{Balister} that any graph with order $n \geq 5$ and size $m > \binom{n-3}{2}+3$ has a path of length $n-1$ and the graphs with size $\binom{n-3}{2}+3$ that have no path of length $n-1$ have the form $K_s \vee (K_{n-2s-2} \cup (s+2)K_1)$ for $s = 1$ or $s = \left \lfloor \frac{n-3}{2} \right \rfloor $. Hence for size $m > \binom{n-3}{2}+3$ the result follows from Corollary~\ref{cor:true for long paths}, and the join graphs with size $\binom{n-3}{2}+3$ with no path of length $n-1$ are easily seen to satisfy the conjecture.
	\end{proof}
	
	In particular, Lemma~\ref{lem:path cover bound} implies that any traceable graph $G$ has $\lgp (G) = 2$. However, the converse does not hold. 
	
	\begin{theorem}
		There are infinitely many non-traceable graphs with lower oriented general position number 2.
	\end{theorem}
	\begin{proof}
		For $1 \leq r_1 \leq r_2 \leq \dots \leq r_t$ and $t \geq 2$, let $\Theta _{r_1,r_2,\dots ,r_t}$ be the theta graph formed from paths $P_{r_1+2},P_{r_2+2}, \ldots ,P_{r_t+2}$ by identifying initial points of the paths to a single vertex $x$ and the terminal points of the paths to a single vertex $y$, so that the graph has order $2+\sum _{i=1}^tr_i$. Consider the oriented theta graph $\Theta _{r,r,s,s}$ obtained by orienting all edges of the two paths of length $r+1$ towards $y$ and all edges of the two paths of length $s+1$ towards $x$. Call each of the oriented paths of length $r+1$ or $s+1$ between $x$ and $y$ an \emph{internal path}. Suppose that this oriented theta graph contains three vertices $z_1,z_2,z_3$ in general position. No internal path can contain all of $z_1,z_2,z_3$. If two of $z_1,z_2,z_3$ lie on the same internal path, say $z_1$ precedes $z_2$ on an internal path, then all $z_1,z_3$-paths pass through $z_2$. Suppose finally that $z_1,z_2,z_3$ all lie on distinct internal paths; without loss of generality $z_1$ and $z_2$ lie on distinct internal paths from $x$ to $y$ and $z_3$ lies on an internal path from $y$ to $x$; observe that there is a shortest $z_1,z_2$-path passing through $z_3$, a contradiction. Therefore these oriented theta graphs have lower oriented general position number 2, but are not traceable. 
	\end{proof}

	Applying the argument of Lemma~\ref{lem:path cover bound} to a longest geodesic in $G$, then the following corollary follows instantly.
	
	\begin{corollary}\label{cor:lgp vs diam}
		For any graph $G$, $\lgp (G) \leq n(G)-\diam (G)+1$.
	\end{corollary}
	
	In particular, if $G$ is not complete, then $\lgp (G) \leq n-1$. In combination with Corollary~\ref{cor:induced bipartite bound}, this yields the following result.
	
	\begin{theorem}\label{thm:conjecturetruecomparability}
		Conjecture~\ref{conj:orientation} is true for comparability graphs.
	\end{theorem}
	
	We mention a few of the best known bounds on the path cover number, which serve as upper bounds for $\lgp (G)$. These will enable us to prove our main conjecture for graphs with small maximum degree. It was conjectured in~\cite{Magnant} that if $G$ is $d$-regular, then $\rho (G) \leq \frac{n}{d+1}$ and in the same paper the conjecture was proven for $d \leq 5$. A stronger result for cubic graphs was given in Reed~\cite{Reed}, namely that the path cover number of a connected cubic graph with order $n$ is at most $\left \lceil \frac{n}{9} \right \rceil $.
	
	\begin{theorem}{\rm \cite{Magnant}}\label{thm:Magnant}
		For $d \leq 5$, if $G$ is $d$-regular, then $\rho (G) \leq \frac{n}{d+1}$.
	\end{theorem}
	
	\begin{theorem}{\rm \cite{Reed}}\label{thm:Reed}
		The path cover number of a connected cubic graph $G$ is bounded above by $\rho (G) \leq \left \lceil \frac{n}{9} \right \rceil $
	\end{theorem}
	
	\begin{theorem}{\rm \cite{Hartman}}
		For any graph $G$ with independence number $\alpha (G)$ and connectivity $\kappa (G)$, $\rho \leq \alpha (G)$, which can be improved to $\rho (G) \leq \alpha (G)-\kappa (G)$ when $\alpha (G) > \kappa (G)$.
	\end{theorem}

	\begin{theorem}\label{thm:max degree 5}
		Conjecture~\ref{conj:orientation} is true for triangle-free subcubic graphs and $d$-regular graphs for $4 \leq d \leq 5$.
	\end{theorem}
	\begin{proof}
		Let $G$ be a $d$-regular graph for some $3 \leq d \leq 5$. Theorem~\ref{thm:Magnant} and Lemma~\ref{lem:path cover bound} show that $\lgp (G) \leq 2\rho (G) \leq \frac{2n}{d+1}$. Theorem~\ref{thm:Alon} and Corollary~\ref{cor:induced bipartite bound} together imply that \[ \ugp (G) \geq \ac (G) \geq n-\frac{m}{4} = \left ( 1-\frac{d}{8}\right ) n. \] For $d \leq 5$ we have $\frac{2}{d+1} < 1-\frac{d}{8}$. The result for subcubic graphs follows similarly from the first part of Theorem~\ref{thm:Alon} and Theorem~\ref{thm:Reed}.
	\end{proof}
	
	Finally, we give some information on the structure of a smallest counterexample to Conjecture~\ref{conj:orientation}.
	
	\begin{theorem}
		If $G$ is a smallest counterexample to Conjecture~\ref{conj:orientation}, then $G$ is 2-connected and has no split.
	\end{theorem}
	\begin{proof}
		Let $G$ be a counterexample to the conjecture with smallest possible order. Suppose that $G$ has a cutvertex $w$. Write the components of $G-w$ as $W_1,\ldots ,W_r$, $r \geq 2$, in non-increasing order of cardinality, and the induced subgraphs $G[V(W_i)\cup \{ w\} ]$ as $G_i$ for $1 \leq i \leq r$. Let there be $t$ components $W_i$ with order $\geq 3$.
		
		By induction the component $W_i$ has orientations with different general position numbers for $1 \leq i \leq t$. Start by orienting the edges of $G$ incident with $w$ away from $w$. If we give each $W_i$ an orientation with general position number $\ugp (W_i)$, then the resulting orientation has general position number at least $\sum _{i=1}^r\ugp (W_i)$ (possibly one more, if $w$ can be added to each general position set of $W_i$ to give a general position set in each $G_i$). If we now orient each $W_i$ to give an orientation with general position number $\lgp (W_i)$, then the orientation has general position number at most $1+\sum _{i=1}^r\lgp (W_i)$. Hence if $t \geq 2$ we have exhibited orientations with different general position numbers. If $t = 1$, alter the second orientation by directing all arcs in $(w,W_r)$ towards $w$. In this case, if $w$ is contained in a largest general position set, then it can be replaced by a vertex in $W_r$, so that again we have strict inequality. Finally, if $t = 0$, then $G$ consists of a vertex connected to copies of $K_1$ and $K_2$, and $G$ is a comparability graph, contradicting Theorem~\ref{thm:conjecturetruecomparability}. 
		
		Now suppose that $G$ has a split, or equivalently $G = G_1 \vee G_2$ for smaller graphs $G_1,G_2$. As the conjecture is true for graphs with order $\leq 6$, we can assume that at least one of $G_1,G_2$, say $G_2$, has order at least 3. By minimality of $G$, there are orientations $G_2^-$ and $G_2^+$ with different general position numbers. Give $G_1$ any orientation, and orient all edges in $(G_1,G_2)$ towards $G_2$. Then by Lemma~\ref{lem:addingsinks} giving $G_2$ the orientations $G_2^-$ and $G_2^+$ will result in orientations of $G$ with distinct general position numbers, a contradiction.
	\end{proof}

	\subsection{General position spectra of families}
	
	In this subsection we compute the $\gp $-spectra of some simple graph families. A case of particular interest is when the orientations of $G$ assume every value allowed by Proposition~\ref{prop1}.
	
	\begin{definition}
		A graph $G$ with order $n$ is \emph{$\gp $-full} if $S_{\gp}(G) = [2,n]$.
	\end{definition}
	
	First we show that complete graphs are $\gp $-full, or, in other words, for any $2 \leq k \leq n$ there exists a tournament with order $n$ and general position number $k$.
	
	\begin{theorem}
		Complete graphs are $\gp $-full.
	\end{theorem}
	\begin{proof}
		We prove the statement by induction on $n$. It is true for $n = 2$.  Assume that the result holds for $K_{n-1}$, i.e.\ for each $k$ with $2 \leq k \leq n-1$, there exists a tournament $D$ such that $\gp(D) = k$.  For each such $D$, add a source vertex adjacent to every vertex of $D$, and the resulting tournament $D'$ has $\gp (D') = k+1$ by Lemma~\ref{lem:addingsinks}. As there is also a tournament $D'$ with order $n$ and $\gp (D') = 2$ by Lemma~\ref{lem:path cover bound}, this completes the induction. 
	\end{proof}
	
	In fact, the following realisation result shows that connected $\gp $-full graphs exist for any feasible order and size.
	
	\begin{theorem}
		For any integers $n$ and $m$ with $n \geq 2$ and $n-1 \leq m \leq \binom{n}{2}$, there exists a connected $\gp$-full graph of order $n$ and size $m$. 
	\end{theorem}
	\begin{proof}
		We prove the slightly stronger statement that for $n \geq 2$ and $m$ in the range $n-1 \leq m \leq \binom{n}{2}$, there is a traceable graph $G$ of order $n$ and size $m$ with spanning path $P$ and an endpoint $x$ of $P$ such that for all $2 \leq k \leq n$, there is an orientation $G_k$ of $G$ with general position number $k$ such that $x$ is a source or a sink. The result is true when $m = n-1$, as the path $P_n$ has the required property by Lemma~\ref{lem:path cover bound}. The result holds for $n = 2$.
		
		Assume that the result is true for all orders less than $n$. Suppose that $2n-3 \leq m \leq \binom{n}{2}$. By induction there is a graph $G'$ with order $n-1$, size $m-(n-1)$ and a spanning path $P'$, such that for all $2 \leq k \leq n-1$, $G'$ has an orientation $G'_k$ with general position number $k$. Add a new universal vertex $x$ to $G'$, yielding a graph $G$ with order $n$, size $m$, and a spanning path $P$ with $x$ as an endpoint. For $2 \leq k \leq n-1$, the orientation $G_k$ consisting of $G'_k$ and source $x$ has general position number $k+1$ by Lemma~\ref{lem:addingsinks}. Also, as $G$ is traceable, it has an orientation with general position number 2 and $x$ as a source by Lemma~\ref{lem:path cover bound}.  
		
		Now suppose that $n-1 \leq m \leq 2n-4$. Let $G'$ be a graph with order $n-1$ and size $m-1$ of the required type, where $P'$ is the spanning path and $x'$ is the endpoint of $P'$ that is a source or sink in each of the orientations $G'_k$. Add a leaf $x$ to $G'$ connected to $x'$ to give a graph $G$ with order $n$, size $m$ and a spanning path $P$ beginning with the edge $xx'$. For $2 \leq k \leq n$, give this graph the orientation $G_k$ corresponding to $G'_k$ and orient the edge $xx'$ as $x' \rightarrow x$ if $x$ is a source and $x \rightarrow x'$ if $x$ is a sink. By the note after Lemma~\ref{lem:addingsinks}, $G_k$ has general position number $k+1$. Again, an orientation of $G$ with general position number 2 and $x$ as a source follows by Lemma~\ref{lem:path cover bound}. 
	\end{proof}
	
	One consequence of note is that for $n \geq 2$ and $n-1 \leq m \leq \binom{n}{2}$ there is a connected graph with order $n$ and size $m$ and an orientation $G^{\rightarrow }$ such that $\gp (G^{\rightarrow }) = \gp (G)$. 
	
	As noted in Subsection~\ref{subsec:trees}, finding the general position number of an oriented tree is difficult in general. We therefore confine ourselves to determining the $\gp $-spectrum of two families of trees. By a \emph{spider} we mean a subdivision of a star in which each edge of the star has been subdivided at least once. Let $x$ be the unique vertex of the spider with degree $\geq 3$.
	
	\begin{theorem}\label{thm:spider}
		If $T$ is a spider with order $n$, then $T$ has $\gp $-spectrum \[
		S_{\gp}(T) = \begin{cases}
			[\ell(T),n] & \text{ if }\deg(x)\text{ is even;}\\
			[\ell(T)+1,n] & \text{ if }\deg(x)\text{ is odd.}\\
		\end{cases}
		\]
	\end{theorem}
	\begin{proof}
		$T$ consists of $\deg (x) = \ell (T)$ paths $P^{(i)}$ of length $\ell (i) \geq 2$ identified at the common endpoint $x$. In any orientation of $T$, $x$ has either $\geq \left \lceil \frac{\ell (T)}{2} \right \rceil $ out-neighbours or $\geq \left \lceil \frac{\ell (T)}{2} \right \rceil $ in-neighbours; we assume the former. Then choosing any two vertices from the paths of $T$ associated with out-neighbours of $x$ yields a general position set of order $\geq 2\left \lceil \frac{\ell (T)}{2} \right \rceil $. If all edges of $\left \lceil \frac{\ell (T)}{2} \right \rceil$ of the paths are oriented away from $x$ and the edges of $\left \lfloor \frac{\ell (T)}{2} \right \rfloor $ of the paths are oriented toward $x$, then this orientation has general position number exactly $\left \lceil \frac{\ell (T)}{2} \right \rceil$. Hence $\lgp (T) \geq \ell (T)+1$ if $\ell (T)$ is odd. 
		
		We start with the case that each path $P^{(i)}$ has length 2. Write the vertex set as $\{ x\} \cup \{ x_i,y_i:1 \leq i \leq \ell (T) \} $, where $x \sim x_i$ and $x_i \sim y_i$ for $1 \leq i \leq \ell (T)$. For $\left \lceil \frac{\ell (T)}{2} \right \rceil \leq t \leq \ell (T) $, let $T_1^{\rightarrow }$ be the orientation with $x \rightarrow x_i \rightarrow y_i$ for $1 \leq i \leq t$ and $y_i \rightarrow x_i \rightarrow x$ for $t+1 \leq i \leq \deg(x)$, and let $T_2^{\rightarrow }$ be the orientation with $x \rightarrow x_i$ and $y_i \rightarrow x_i$ for $1 \leq i \leq t$ and $y_i \rightarrow x_i \rightarrow x$ for $t+1 \leq i \leq \ell (T)$. Then it is simple to verify that $\gp (T_1^{\rightarrow }) = 2t$ and $\gp (T_2^{\rightarrow }) = 2t+1$.
		
		We now use induction on the order of the spider. Suppose that the result is true for all spiders with order $< n$ and let $T$ be a spider with order $n$. We can assume that $T$ has a path from $x$ of length at least 3; let $z$ be a leaf vertex of such a path with support vertex $z'$. By induction for $\left \lceil \frac{\ell (T)}{2} \right \rceil \leq k \leq n-1$ the spider $T-z$ has an orientation $T^k$ with general position number $k$. For each $k$ in this range, we add the edge $z'z$ to $T^k$ and orient the edge $z'z$ so that $z'$ is either a source or a sink. By the note after Lemma~\ref{lem:addingsinks}, the resulting orientation has general position number $k+1$. As $T$ has an orientation with general position number $\left \lceil \frac{\ell (T)}{2} \right \rceil $, this has established existence for all required values of the general position number, and the result follows by induction.
	\end{proof}
	
	A \emph{caterpillar} is a tree consisting of a path $P_r$ called the \emph{spine} and leaves attached to the internal vertices of the spine.
	
	\begin{theorem}\label{thm:caterpillar}
		If $T$ is a caterpillar with order $n$, then the $\gp $-spectrum of $T$ is $[\ell (T),n]$.
	\end{theorem}
	\begin{proof}
		Denote the vertices of the path $P_r$ by $x_1,x_2,\ldots ,x_r$ and suppose that there are $s$ additional leaves attached to $x_2,\ldots ,x_{r-1}$, so that $n = r+s$. By Lemma~\ref{lem:leaf bound} we have $\lgp (T) \geq \ell (T) = s+2$. As $T$ has diameter $r-1$, Corollary~\ref{cor:lgp vs diam} shows that $\lgp (T) \leq (r+s)-(r-1)+1 = s+2$. Therefore $\lgp (T) = \ell (T)$. This shows that the result holds when $T$ is a star, so we can assume that $\diam (T) \geq 3$.
		
		We now use induction on the order of the tree. Suppose that the result is true for all trees of order $< n$ and $T$ be a tree with order $n$. If $T$ is a star, we are done, otherwise let $L$ be the set of leaves of $T$ adjacent to $x_2$, including $x_1$. Then by induction the tree $T' = T-L$ has orientations for all $\gp $-numbers between $\ell (T') = \ell (T)-|L|+1$ and order $n(T') = n(T)-|L|$. In each of these orientations $x_2$ will be a source or a sink, so by the note after Lemma~\ref{lem:addingsinks} the set $L$ can be added back to each of these orientations to give an orientation of $T$ with $\gp $-number anywhere in the range $\ell (T)+1$ and $n(T)$. As the existence of an orientation with $\gp $-number $\ell (T)$ has been established, the proof is complete.
	\end{proof}

	\begin{theorem}
		The $\gp $-spectrum of a cycle $C_n$ is given by $[2,n]$ if $n$ is even and $[2,n-1]$ if $n$ is odd. The general position number of the directed cycle $Z_n$ is $\gp (Z_n) = 2$ for $n \geq 3$.
	\end{theorem}
	\begin{proof}
		We identify the vertex set of $C_n$ with $\mathbb{Z}_n$ in the natural fashion. As odd cycles are not comparability graphs but even cycles are, it follows from Corollary~\ref{cor:comparability} that $\ugp (C_n) \leq n-1$ when $n$ is odd and $\ugp (C_n) = n$ for even $n$. Alternatively, observe that if $n$ is odd there must be two consecutive arcs with the same direction, and so $\ugp (C_n) \leq n-1$ follows from Lemma~\ref{lem:diambound}, whereas if $n$ is even the arcs can be oriented alternately to make each vertex a source or sink, and $\ugp (C_n) = n$ follows from Lemma~\ref{lem:ext}.  
		
		For each $1 \leq j \leq n-1$ with $j \equiv n+1 \pmod{2}$, we denote by $Z_n^{(j)}$ the oriented cycle with arcs $(i,i+1)$ for $0 \leq i \leq j-1$ and the other arcs oriented alternately so that the vertices of $\mathbb{Z}_n-[1,j-1]$ are sources or sinks. Then $\gp (Z_n^{(j)}) \geq n-j+1$ by Lemma~\ref{lem:ext} and since $\diam ^*(Z_n^{(j)}) = j$ we have equality by Lemma~\ref{lem:diambound}.
		
		Now for each $1 \leq j \leq n-1$ and $j \equiv n \pmod{2}$, let $Z_n^{(j)}$ be the oriented cycle with arcs $(0,n-1)$, $(n-1,n-2)$ and $(i,i+1)$ for $0 \leq i \leq j-1$ and all remaining edges oriented to make the vertices in $[j,n-2]$ sources or sinks. The set $[j-1,n-1]$ is in general position, giving $\gp (Z_n^{(j)}) \geq n-j+1$, and again we have equality by Lemma~\ref{lem:diambound}.
		
		The final statement follows from Lemma~\ref{lem:path cover bound}.
	\end{proof}
	
	The $\gp $-spectrum of complete multipartite graphs is difficult to find in general. However we show that if the graph is very unbalanced, then the lower oriented general position number can be arbitrarily large. We start with a lower bound on $\lgp (G)$ that applies for all graphs. Two vertices $u,v$ of an undirected graph $G$ are \emph{twins} if $N(u) = N(v)$. A \emph{twin set} of $G$ is a subset $S \subset V(G)$ such that all vertices of $S$ have the same neighbourhood $N(S)$. We extend this to directed graphs $D$ by defining vertices $u,v$ to be twins if $N^+(u) = N^+(v)$ and $N^-(u) = N^-(v)$, and a subset $S \subset V(D)$ is a twin set if all vertices of $S$ have the same out-neighbourhood $N^+(S)$ and the same in-neighbourhood $N^-(S)$. A twin set is easily seen to be in general position, in both the undirected and directed settings.
	
	\begin{lemma}
		If $G$ has a twin set $S$ with common neighbourhood $N(S)$, then \[ \lgp (G) \geq \frac{|S|}{2^{|N(S)|}}.\]   
	\end{lemma}
	\begin{proof}
		Let $S$ be a twin set of $G$ with common neighbourhood $N(S)$. For each vertex $u \in S$, there are $2^{|N(S)|}$ possible orientations of the edges from $u$ to $N(S)$. By the Pigeonhole Principle, there are at least $\frac{|S|}{2^{|N(S)|}}$ vertices of $S$ with the same orientations of edges to $N(S)$ and these vertices are a twin set in the oriented graph. 
	\end{proof}
	
	\begin{corollary}
		For $r_1 \geq \ldots \geq r_t \geq 2$, $t \geq 2$, the lower oriented general position number of $K_{r_1,\ldots ,r_t}$ is bounded below by \[ \lgp (K_{r_1,\ldots ,r_t}) \geq \frac{r_1}{2^{r_2+\ldots +r_t}}.\] 
	\end{corollary}
	
	\begin{theorem}
		For $r \geq 4$, the $\gp $-spectrum of $K_{r,2}$ is $[\left \lceil \frac{r}{2} \right \rceil ,r+2]$.
	\end{theorem}
	\begin{proof}
		Denote the set of vertices of the part of order $r$ by $X = \{ x_1,\ldots ,x_r\} $ and the vertices of the smaller part by $y,z$. Let $K_{r,2}^{\rightarrow }$ be any orientation of $K_{r,2}$. Taking the converse if necessary, we can assume that in $K_{r,2}^{\rightarrow }$ we have $|N^+(y)| \geq   \left \lceil \frac{r}{2} \right \rceil $. We claim that $N^+(y)$ is in general position. This is trivial if $r \leq 4$, so we set $r \geq 5$. Suppose that there is a shortest $w_1,w_3$-path $P$ through $w_2$, where $w_1,w_2,w_3 \in N^+(y)$. Then $P$ must begin $w_1,z,w,y$, where $w \in N^-(y)$; however, $y$ is adjacent to both $w_2$ and $w_3$, contradicting our assumption that $P$ is a geodesic. Hence all orientations of $K_{r,2}$ have general position number at least $\left \lceil \frac{r}{2} \right \rceil $.
		
		Let $\left \lceil \frac{r}{2} \right \rceil \leq t \leq r$ and consider the orientation in which $y \rightarrow x_i$ and $x_i \rightarrow z$ for $1 \leq i \leq t$ and $x_j \rightarrow y$ and $z \rightarrow x_j$ for $t+1 \leq j \leq r$. If $t = r$, then the orientation has general position number $r+1$ ($X \cup \{ y\} $ is in general position), whilst the upper bound follows from Lemma~\ref{lem:diambound}, so we assume that $t < r$. For $1 \leq i_1,i_2 \leq t$ and $t+1 \leq j_1,j_2 \leq r$ there are shortest paths $x_{i_1},z,x_{j_1},y,x_{i_2}$ and $x_{j_1},y,x_{i_1},z,x_{j_2}$, and a routine argument shows that the largest general position set has order $t$. 
		
		Finally, the orientation in which $x_1,\ldots ,x_{r-1}$ are sinks and $x_r$ is a source has general position number $r$, whilst there is an orientation with general position number $r+2$ by Corollary~\ref{cor:comparability}.
	\end{proof}
	
	\begin{theorem}
		The complete bipartite graph $K_{n,n}$ is $\gp $-full for $n \geq 2$.
	\end{theorem}
	\begin{proof}
		Write the vertex set of $K_{n,n}$ as $X \cup Y$, where $X = \{ x_1,\ldots ,x_n\} $ and $Y= \{ y_1,\ldots ,y_n\} $. For even $0 \leq r \leq n-2$, define the orientation $K^{(1)}(n)_r$ of $K_{n,n}$ as follows: the arcs directed from $Y$ to $X$ are:
		\begin{itemize}
			\item $y_i \rightarrow x_i$ for $1 \leq i \leq n$,
			\item $y_i \rightarrow x_j$ when $r+1 \leq j \leq i-2$, 
			\item $y_i \rightarrow x_j$ when $1 \leq i \leq r$, $i$ is odd and $j \geq i+2$, and
			\item $y_i \rightarrow x_j$ when $1 \leq i \leq r$, $i$ is even and $j \geq i+1$.
		\end{itemize}
		All other arcs are directed from $X$ to $Y$. Observe that for odd $i \leq r$, the subdigraph induced by $\{ x_i,x_{i+1},y_i,y_{i+1}\} $ is isometric, as is the subdigraph induced by $\{ x_{r+1},\ldots ,x_n\} \cup \{ y_{r+1},\ldots ,y_n\} $. Each of these subdigraphs has general position number 2 (observing that the last of these has the orientation defined by Lemma~\ref{lem:path cover bound}), so by Theorem~\ref{thm:isometric cover} we have $\gp (K^{(1)}(n)_r) \leq r+2$. The set $\{ x_1,\ldots ,x_r\} \cup \{ y_{r+1},y_{r+2}\} $ is in general position, so in fact $\gp (K^{(1)}(n)_r) = r+2$. This covers all even values in the range $[2,n]$. A similar idea works for the odd values. In the previous construction, when $0 \leq r \leq n-3$ is even, change the direction of the arc $x_{r+1} \rightarrow y_{r+2}$ to $y_{r+2} \rightarrow x_{r+1}$. Lemma~\ref{lem:path cover bound} and Theorem~\ref{thm:isometric cover} show that the general position number of this new orientation is at most $r+3$, and $\{ x_1,\ldots ,x_{r+2}\} \cup \{ y_{r+1}\} $ is a general position set of that order.
		
		For even $r$ in the range $0 \leq r \leq n-2$ we now derive an orientation $K^{(2)}(n)_r$ from $K^{(1)}(n)_r$ by orienting each arc between $\{ y_{r+1},\ldots ,y_n\} $ and $\{ x_{r+1},\ldots ,x_n\} $ towards $X$, i.e.\ $y_i \rightarrow x_j$ is an arc for $r+1 \leq i,j \leq n$. The subdigraph induced by $\{ x_{r+1},\ldots ,x_n\} \cup \{ y_{r+1},\ldots ,y_n\} $ is isometric and has general position number $2(n-r)$ (as each vertex is a source or sink), and so Theorem~\ref{thm:isometric cover} gives $\gp (K^{(2)}(n)_r) \leq 2n-r$, and $X \cup \{ y_{r+1},\ldots ,y_n\} $ is a general position set, so that equality holds. This gives all even values of the general position number in the range $[n+1,2n]$. As in the previous case, changing the arc $y_{r+1} \rightarrow x_{r+1}$ to $x_{r+1} \rightarrow y_{r+1}$ deals with the odd values in the range.   
	\end{proof}

	For all the graphs for which we have computed the $\gp $-spectrum, the spectrum is an interval. We conjecture that this holds generally.
	
	\begin{conjecture}
		The general position spectrum of any graph is an interval.
	\end{conjecture}

	\section{Conclusion}\label{sec:conclusion}
	
	We conclude with some open problems suggested by this research. One possible direction for investigation is to restrict the orientations considered in Conjecture~\ref{conj:orientation}. It is easy to find families of graphs such that all strong orientations, or all source- and sink-free orientations, have the same general position number. 
	
	\begin{problem}
		Characterise the graphs with the property that all strong orientations (or all source- and sink-free orientations) have the same general position number. 
	\end{problem}
	
	It would be especially interesting to find exact values for trees, permutation digraphs and Kautz digraphs for $k \geq 3$.
	
	\begin{problem}
		Find a formula for the general position number of any tree, as well as permutation digraphs and Kautz digraphs for $k \geq 3$.
	\end{problem}
	
	Regarding complexity, we have shown in Theorem~\ref{thm:NP} that finding the general position number of an oriented graph is NP-complete. However, except for the fact that $\ugp (G) = n$ can be tested in polynomial time, we have not considered the complexity of finding the $\gp $-spectrum. We conjecture that it is hard even to find the lower oriented general position number. 
	
	\begin{conjecture}
		It is NP-hard to decide if a graph $G$ satisfies $\lgp (G) = 2$.
	\end{conjecture}
	
	There are many variants of the general position problem, including monophonic position and mobile general position~\cite{survey}. We suggest that it would be rewarding to investigate directed versions of these problems. Finally, we note that a logical next step would be to define and investigate the general position problem in signed graphs.
	
	\begin{problem}
		Investigate the general position problem for signed graphs.
	\end{problem}

	\section*{Acknowledgements}
	Haritha S thanks the University of Kerala for providing financial support under the University JRF Scheme.


\end{document}

%% file: Ka52.tex
\begin{tikzpicture}[x=0.2mm,y=0.2mm,very thick,vertex/.style={circle,draw,minimum size=10,inner sep=1pt,fill=white,font=\scriptsize}]
	\node at (42.7,266.6) [vertex,fill=red] (v1) {$12$};
	\node at (275,-53.1) [vertex,fill=red] (v2) {$13$};
	\node at (-44.7,-285.3) [vertex] (v3) {$14$};
	\node at (-277,34.4) [vertex] (v4) {$15$};
	\node at (-44.7,266.6) [vertex] (v5) {$21$};
	\node at (125.8,239.6) [vertex] (v6) {$23$};
	\node at (125.8,-258.3) [vertex] (v7) {$24$};
	\node at (-198.6,-206.8) [vertex] (v8) {$25$};
	\node at (-277,-53) [vertex] (v9) {$31$};
	\node at (-127.9,239.6) [vertex] (v10) {$32$};
	\node at (196.6,188.2) [vertex] (v11) {$34$};
	\node at (247.9,-136.2) [vertex] (v12) {$35$};
	\node at (42.6,-285.3) [vertex] (v13) {$41$};
	\node at (-127.8,-258.3) [vertex,fill=red] (v14) {$42$};
	\node at (-198.5,188.3) [vertex,fill=red] (v15) {$43$};
	\node at (247.9,117.5) [vertex] (v16) {$45$};
	\node at (275,34.4) [vertex,fill=red] (v17) {$51$};
	\node at (196.5,-206.9) [vertex,fill=red] (v18) {$52$};
	\node at (-250,-136.2) [vertex,fill=red] (v19) {$53$};
	\node at (-249.9,117.5) [vertex,fill=red] (v20) {$54$};
	\draw [<->] (v1) to (v5);
	\draw [->] (v1) to (v6);
	\draw [->] (v1) to (v7);
	\draw [->] (v1) to (v8);
	\draw [<->] (v2) to (v9);
	\draw [->] (v2) to (v10);
	\draw [->] (v2) to (v11);
	\draw [->] (v2) to (v12);
	\draw [<->] (v3) to (v13);
	\draw [->] (v3) to (v14);
	\draw [->] (v3) to (v15);
	\draw [->] (v3) to (v16);
	\draw [<->] (v4) to (v17);
	\draw [->] (v4) to (v18);
	\draw [->] (v4) to (v19);
	\draw [->] (v4) to (v20);
	\draw [->] (v5) to (v2);
	\draw [->] (v5) to (v3);
	\draw [->] (v5) to (v4);
	\draw [->] (v6) to (v9);
	\draw [<->] (v6) to (v10);
	\draw [->] (v6) to (v11);
	\draw [->] (v6) to (v12);
	\draw [->] (v7) to (v13);
	\draw [<->] (v7) to (v14);
	\draw [->] (v7) to (v15);
	\draw [->] (v7) to (v16);
	\draw [->] (v8) to (v17);
	\draw [<->] (v8) to (v18);
	\draw [->] (v8) to (v19);
	\draw [->] (v8) to (v20);
	\draw [->] (v9) to (v1);
	\draw [->] (v9) to (v3);
	\draw [->] (v9) to (v4);
	\draw [->] (v10) to (v5);
	\draw [->] (v10) to (v7);
	\draw [->] (v10) to (v8);
	\draw [->] (v11) to (v13);
	\draw [->] (v11) to (v14);
	\draw [<->] (v11) to (v15);
	\draw [->] (v11) to (v16);
	\draw [->] (v12) to (v17);
	\draw [->] (v12) to (v18);
	\draw [<->] (v12) to (v19);
	\draw [->] (v12) to (v20);
	\draw [->] (v13) to (v1);
	\draw [->] (v13) to (v2);
	\draw [->] (v13) to (v4);
	\draw [->] (v14) to (v5);
	\draw [->] (v14) to (v6);
	\draw [->] (v14) to (v8);
	\draw [->] (v15) to (v9);
	\draw [->] (v15) to (v10);
	\draw [->] (v15) to (v12);
	\draw [->] (v16) to (v17);
	\draw [->] (v16) to (v18);
	\draw [->] (v16) to (v19);
	\draw [<->] (v16) to (v20);
	\draw [->] (v17) to (v1);
	\draw [->] (v17) to (v2);
	\draw [->] (v17) to (v3);
	\draw [->] (v18) to (v5);
	\draw [->] (v18) to (v6);
	\draw [->] (v18) to (v7);
	\draw [->] (v19) to (v9);
	\draw [->] (v19) to (v10);
	\draw [->] (v19) to (v11);
	\draw [->] (v20) to (v13);
	\draw [->] (v20) to (v14);
	\draw [->] (v20) to (v15);
\end{tikzpicture}

%% file: tree7gp5.tex
\begin{tikzpicture}[x=0.2mm,y=0.2mm,very thick,vertex/.style={circle,draw,minimum size=10,fill=lightgray}]
	\node at (80,180) [vertex] (v1) {};
	\node at (-80,180) [vertex] (v2) {};
	\node at (-220,120) [vertex] (v3) {};
	\node at (-220,240) [vertex] (v4) {};
	\node at (160,180) [vertex] (v5) {};
	\node at (0,180) [vertex] (v6) {};
	\node at (-160,180) [vertex] (v7) {};
	\node at (80,0) [vertex] (v8) {};
	\node at (-80,0) [vertex] (v9) {};
	\node at (-220,60) [vertex] (v10) {};
	\node at (-220,-60) [vertex] (v11) {};
	\node at (160,0) [vertex] (v12) {};
	\node at (0,0) [vertex] (v13) {};
	\node at (-160,0) [vertex] (v14) {};
	\draw [->] (v1) to (v6);
	\draw [->] (v2) to (v6);
	\draw [->] (v3) to (v7);
	\draw [->] (v5) to (v1);
	\draw [->] (v7) to (v2);
	\draw [->] (v7) to (v4);
	\draw [->] (v8) to (v12);
	\draw [->] (v9) to (v14);
	\draw [->] (v10) to (v14);
	\draw [->] (v13) to (v8);
	\draw [->] (v13) to (v9);
	\draw [->] (v14) to (v11);
\end{tikzpicture}